\documentclass[a4paper,12pt,reqno]{amsart}
\usepackage{amssymb,amsmath,amsthm}
\usepackage[left=.75 in, right=.75 in,top=.75 in, bottom=.75 in]{geometry}
\allowdisplaybreaks
\usepackage{graphicx}
\usepackage{tikz}
\usepackage{float}
\usepackage{enumitem}
\usepackage{changepage}
\usepackage{stackengine}
\usepackage{hyperref}
\usepackage{subcaption}
\usepackage{cases}
\usepackage{amsaddr}

\def\p{\partial}

\def\H1{{_0}H^1(\Omega)}

\def\N{\mathbb{N}}

\def\R{\mathbb{R}}

\def\Xint#1{\mathchoice
{\XXint\displaystyle\textstyle{#1}}%
{\XXint\textstyle\scriptstyle{#1}}%
{\XXint\scriptstyle\scriptscriptstyle{#1}}%
{\XXint\scriptscriptstyle\scriptscriptstyle{#1}}%
\!\int}
\def\XXint#1#2#3{{\setbox0=\hbox{$#1{#2#3}{\int}$ }
\vcenter{\hbox{$#2#3$ }}\kern-.6\wd0}}
\def\intbar{\Xint-}

\newtheorem{lem}{Lemma}[section]

\newtheorem{prop}[lem]{Proposition}
\newtheorem{thm}[lem]{Theorem}
\newtheorem{remark}[lem]{Remark}
\newtheorem{definition}[lem]{Definition}

\newtheorem*{prop*}{Proposition}
\newtheorem*{thm*}{Theorem}
\newtheorem*{def*}{Definition}
\newtheorem*{lem*}{Lemma}

\newcommand{\eps}{\epsilon}
\newcommand{\ov}{\overline}

\newcommand{\ra}{\rangle}
\newcommand{\la}{\langle}

\newcommand{\wkto}{\rightharpoonup}
\newcommand{\wksto}{\overset{*}{\rightharpoonup}}

\newcommand {\jump}[1] {[\![ #1 ]\!]}

\numberwithin{equation}{section}

\begin{document}
\pagestyle{plain}
\setcounter{page}{1}

\title[Sharp Interface Limit of the Cahn--Hilliard Reaction Model]
{Sharp Interface Limit of the Cahn--Hilliard Reaction Model for Lithium-ion Batteries}

\author{Tim Laux}
\author{Kerrek Stinson}
\address{Hausdorff Center for Mathematics, University of Bonn, Endenicher Allee 62, 53115 Bonn, Germany}
\email{tim.laux@hcm.uni-bonn.de}
\email{kerrek.stinson@hcm.uni-bonn.de}

\begin{abstract} 
We propose a weak solution theory for the sharp interface limit of the Cahn--Hilliard reaction model, a variational PDE for lithium-ion batteries. An essential feature of this model is the use of Butler--Volmer kinetics for lithium-ion insertion, which arises as a Robin-type boundary condition relating the flux of the chemical potential to the reaction rate, itself a nonlinear function of the chemical potential and the ion concentration.
To pass through the nonlinearity as interface width vanishes, we introduce solution concepts at the diffuse and sharp interface level describing dynamics principally in terms of an optimal dissipation inequality. Using this functional framework and under an energy convergence hypothesis, we show that solutions of the Cahn--Hilliard reaction model converge to a Mullins--Sekerka type geometric evolution equation.



\end{abstract}


\maketitle

\section{Introduction}

Lithium-ion batteries play a critical role in the move towards renewable energies. However, batteries are far from reaching their potential and come with many engineering challenges. A fundamental short-coming of current lithium-ion batteries is a limited life-cycle. In the process of charging, lithium-ions intercalate into the host material of the cathode, e.g., iron phosphate FePO$_4$, heterogeneously thereby forming distinct regions of high and low lithium-ion concentration. Large concentration gradients induce a strain in the host material which can ultimately cause damage to the cathode. Unsurprisingly, repeated use will come with a decrease in performance (see \cite{Bazant-Theory2013,Dal2015-Comp}, and references therein). To overcome this challenge, the development and analysis of phenomenologically accurate and physically consistent models for phase separation within batteries are essential.

We consider the Cahn--Hilliard reaction (CHR) model
\begin{align}
\label{pde:CHRnoElastic}
\text{CHR}_\epsilon \text{ model} \ &&\left\{ \begin{aligned}
& \partial_t c = \Delta \mu  &&\text{in } \Omega \times (0,T),\\
&  \mu = -\epsilon\Delta c +\frac{1}{\epsilon}f'(c)  &&\text{in } \Omega \times (0,T),\\
& \partial_n c = 0 && \text{on }\partial \Omega \times (0,T), \\
& \partial_n \mu = R(c,\mu ) &&\text{on }\partial \Omega \times (0,T), \\
& c(0) = c_0 && \text{in } \Omega
\end{aligned}\right.
\end{align}
introduced by Bazant, et al.\ (see \cite{Bazant-Theory2013,burch2009size,cogswell2012coherency,singh2008intercalation,Bazant-PhaseSepDyn2014}) as an electrochemically consistent model for phase separation within a nanoparticle of a lithium-ion battery.
Here $\Omega\subset \R^N$ is a fixed domain representing the nanoparticle, $c$ is the normalized lithium-ion density, the length-scale $0<\epsilon\ll 1$ characterizes the typical width of transition layers between regions of almost-constant lithium-ion density, $f$ is a potential associated with chemical mixing, and $R$ represents the inclusion of Butler--Volmer kinetics for the insertion of lithium-ions into the domain. In contrast to other modeling paradigms, the CHR model can incorporate essential material behavior such as anisotropy \cite{Bazant-Theory2013}, elasticity \cite{cogswell2012coherency}, and damage \cite{kraus2017,O_Connor_2016}.

The purpose of this paper is to develop a solution theory for the sharp interface limit of the CHR$_\epsilon$ model, which qualitatively elucidates how the geometry of the interface between regions of low and high lithium-ion density drives the process of phase separation and coarsening. Using formal asymptotic expansions, this question has been considered for the related viscous CHR model \cite{meca_munch_wagner_2018}, and our result will rigorously validate this by identifying the limit as a Mullins--Sekerka type equation with a boundary reaction term accounting for the insertion of lithium-ions.

Within materials science, in both static (see, e.g., \cite{ambrosio-tortorelli-1990,ContiFonsecaLeoni-gammConv2grad,FJMrigid,ModicaMortola,stinsonLiBatteryGamma}) and evolutionary settings (e.g., \cite{Abels2015SharpIL,chen1996,sandierSerfaty-gammaGrad,Serfaty2011}), insight for energetic models can often be derived via their asymptotic limit. 
The classical result of Modica \cite{Modica87} (see also \cite{ModicaMortola}) shows that for a domain $\Omega\subset \R^N,$ the Cahn--Hilliard free energy 
\begin{equation}\label{def:energy}
I_\epsilon[c] : = \int_\Omega \Big( \frac{1}{\epsilon} f(c) + \frac{\epsilon}{2}\|\nabla c\|^2 \Big)\, dx
\end{equation}
converges to a scaled perimeter/area functional 
\begin{equation}\nonumber
I_0[c_0]  = \sigma {\rm Area}(\partial \{c_0=1\} \cap \Omega) , \quad \text{where }c_0 \colon \Omega \to \{0,1\},
\end{equation}
with $f$ a two-well potential, e.g., $f(s) = s^2(s-1)^2$, and $\sigma > 0$ the surface tension energy density (see (\ref{def:energyPer})).
This means that the a priori complicated phase transition effectively behaves like a minimal surface when near an equilibrium. 
Looking to the chemical potential (later called $\mu$) associated with the first variation of $I_\epsilon$, Luckhaus and Modica \cite{LuckhausModica} verified a conjecture of Gurtin's regarding the Gibbs--Thomson relation, thereby showing that the curvature of the underlying minimal surface closely approximates the chemical potential of energy minimizers. Though the relation was explicitly derived for minimizers of $I_\epsilon$, these tools are often helpful in the study of curvature driven evolution equations (see Section \ref{sec:sharpInterface} and also, e.g., \cite{chen1996,kroemerLaux2021}).

Solutions of the CHR$_\epsilon$ model are realized in a variational framework as solutions of a doubly nonlinear differential inclusion given by
\begin{equation}\label{eqn:doublyNonlinearDiffInclusion}
0 \in \partial I_\epsilon[c] - \partial \mathcal{A}^*_{c}(-\partial_t c),
\end{equation}
where $\mathcal{A}^*_{c}$ operates as a metric for the gradient flow that encodes the Robin-type boundary condition $\partial_n \mu = R(c,\mu)$ (see Section \ref{sec:funcFramework}). In the case of the typical Cahn--Hilliard equation, $\mathcal{A}^*(z) = \|z\|_{H^{-1}_{(0)}}^2,$ where $H^{-1}_{(0)}$ is the dual of $H^1_{(0)}:=H^1(\Omega)\cap \{u:\int_\Omega u =0\}$ the Sobolev space with $0$-mass average. Here (\ref{eqn:doublyNonlinearDiffInclusion}) reduces to the standard gradient flow form 
\begin{equation}\label{eqn:diffInclusionCH}
-\partial_t c \in  \partial I_\epsilon[c],
\end{equation}
 where the subdifferential has values in $H^{-1}_{(0)}.$

At the fixed $\epsilon$-level, gradient flows in Banach spaces such as (\ref{eqn:diffInclusionCH}) are well understood, and much work in the past 30 years has been dedicated to expanding the power and generalizing gradient flow perspectives to weaker settings such as for metric spaces, optimal transport, or in relation to $\Gamma$-convergence (see, e.g., \cite{AGS-GradFlows,sandierSerfaty-gammaGrad,Serfaty2011}). 
For doubly nonlinear differential inclusions like (\ref{eqn:doublyNonlinearDiffInclusion}), we refer to Mielke, et al. \cite{mielke-doublyNonlinearEvo}, and with specific application to the CHR$_\epsilon$ model, we refer the reader to \cite{kraus2017}, \cite{stinson-phd}, and Section \ref{sec:CHRexist}, where we prove existence of weak solutions.

Though the differential inclusion (\ref{eqn:diffInclusionCH}) at the $\epsilon$-level is expressed with a linear metric structure given by $H^{-1}_{(0)},$ the limiting equation, commonly known as the Mullins--Sekerka equation (or flow), is a curvature driven geometric evolution equation over a formal Hilbert-manifold. In the seminal work of Alikakos, et al.\ \cite{AlikakosBatesChen}, they proved convergence as $\epsilon\to 0$ of solutions of the Cahn--Hilliard equation to the Mullins--Sekerka flow assuming that a smooth solution of the limiting equation exists. However to understand the equation in the weak setting, for instance after topological change, an alternative approach is needed. 

To prove convergence of solutions of the Cahn--Hilliard equation to a weak solution of the Mullins--Sekerka flow, Chen introduced estimates on (the positive part of) the discrepancy measure, which captures the lack of an equiparition of energy \cite{chen1996}. A principle challenge is that, in contrast to the second-order Allen--Cahn equation \cite{ilmanen}, the maximum principle is no longer valid. Though this approach has proven widely applicable (e.g., \cite{abelsLengeler2014,AbelsRoeger,melchionnaRocca2017}), it requires the introduction of an additional varifold (``weighted interface") to account for the Gibbs--Thomson relation that is only valid in the domain interior, and consequently cannot account for the contact angle of the interface with the domain boundary. Further, as the evolution is given by a gradient flow, it would be natural to express the dynamics in terms of an optimal dissipation inequality.

A powerful alternative to expressing the nonlinear structure through an explicit differential inclusion as in (\ref{eqn:diffInclusionCH}) is the approach of De Giorgi (see \cite{AGS-GradFlows} and references therein; cf. Sandier and Serfaty \cite{sandierSerfaty-gammaGrad}), which was already partly envisioned by Onsager \cite{onsager31}. Precisely, one interprets solutions as curves of maximal slope, that is, curves that satisfy an optimal energy dissipation relation. This concept was introduced by De Giorgi, et al.\ for metric spaces \cite{DeGiorgi1980}, but has also recently found its footing for mean-curvature flows \cite{Hensel2021l}, where it is well known that the metric (geodesic) structure is degenerate. Beyond unconditional existence, in \cite{Hensel2021l}, the optimal dissipation solution concept is powerful enough to provide a weak-strong uniqueness principle for mean curvature flow. With application to the Mullins--Sekerka equation, a related solution concept was developed in the smooth setting by Le \cite{Le2008} and recently extended to the weak setting by Hensel and the second author in \cite{henselStinson-weakSolnMS}. 

To capture the optimal dissipation for the CHR$_\epsilon$ model as $\epsilon\to 0$, there will be two nonlinearities to account for in (\ref{eqn:doublyNonlinearDiffInclusion}): the first is the Gibbs--Thomson relation expressing the curvature of the evolving surface by the trace of the chemical potential, just as in the case of the classical Cahn--Hilliard equation; the second is the nonlinear metric structure arising from the Butler--Volmer kinetics (in the spirit of Mielke, et al.\ \cite{mielke-doublyNonlinearEvo}). To account for the coupling of these nonlinearities, we will not express the velocity explicitly by an $H^1(\Omega)$ potential (i.e., $\partial_t c  = \Delta \mu$) satisfying the Gibbs--Thomson relation, as is typically done for Mullins--Sekerka (see, e.g., \cite{chen1996,LucStu,Roeger2005}), but instead couple the velocity with a potential representing the Gibbs--Thomson relation only through the optimal dissipation relation, consistent with \cite{Hensel2021l} and \cite{henselStinson-weakSolnMS}.

The principal result of the paper is given in Theorem \ref{thm:BVsolnExist}, which we summarize here.
\begin{thm*}
Any sequence of (weak) solutions $c_\eps$ of the CHR$_\epsilon$ model (\ref{pde:CHRnoElastic}) with well-prepared initial conditions has a subsequence converging to some limit $c \colon \Omega\times [0,T) \to \{0,1\}$ as $\eps\to0$.
Under an energy convergence hypothesis, the limit $c$ is a weak solution of the Mullins--Sekerka reaction model, a geometric evolution equation (see Definition \ref{def:strongSolnMSR}).
\end{thm*}

The paper is outlined as follows. In Section \ref{sec:funcFramework}, we develop a functional framework that is applicable to the variational structure of the CHR$_\epsilon$ model and the limiting sharp interface model.
Though not the principle motivation of the paper, it is of interest that in Section \ref{sec:CHRexist} we prove existence of weak solutions for the CHR$_\epsilon$ with $\epsilon>0$ fixed. To do this, we introduce a solution concept based on the optimal dissipation, which under mild-regularity assumptions is consistent with more traditional distributional formulations as used in \cite{kraus2017} and \cite{stinson-phd}. Finally, in Section \ref{sec:sharpInterface}, we introduce strong and weak formulations of the Mullins--Sekerka reaction model. 
Therein, we prove that smooth solutions of the weak formulation are in fact strong solutions (Theorem \ref{thm:consistency}). Further, we obtain precompactness of the solution space as $\epsilon \to 0$ and show that solutions of the CHR$_\epsilon$ model converge to weak solutions of the Mullins--Sekerka reaction model in Theorem~\ref{thm:BVsolnExist}.

The primary contribution of this paper to the existing literature on models for lithium-ion batteries is the rigorous derivation of the effective behavior of phase-field geometry on the dynamics of coarsening, confirming the sharp interface limit identified for a related viscous CHR$_\epsilon$ model using formal asymptotic expansions in \cite{meca_munch_wagner_2018}. The primary mathematical contribution is the refinement of a functional framework and introduction of solution concepts that account for the nonlinear Robin-type boundary condition $\partial_n \mu = R(c,\mu)$ and capture the dynamics of the diffuse and sharp interface models via an optimal dissipation inequality.

\section{Assumptions}

We will make use of a collection of assumptions to prove existence of weak solutions to the CHR$_\epsilon$ model. In application \cite{singh2008intercalation}, the potential $f$ is given by a logarithmic two-well function and the reaction $R$ is an exponential type function. As is commonly done, we will make the non-physical assumption that $f$ and $R$ have polynomial growth. We also include some structural assumptions on $f$ and $R$, but these are satisfied in application as well (cf. \cite{singh2008intercalation}). We remark that the truncation type arguments of \cite{elliott-CHdegenMobility} might possibly be used to remove the polynomial-growth assumption on $f$; however, it is unlikely such an approach would carry over for $R$ (see also \cite{stinsonLiBatteryExpBC} and Lemma \ref{lem:H1dualbdd}).

  For convenience, one can always think of $2^\#-\delta$, introduced below, as equal to $2$.

\begin{enumerate}[label=(\roman*)]
\item For some $0<\delta \leq 2^\#-2$, we assume that the chemical energy density is governed by a function $f\in C^2(\R)$ such that, for some $C>0$,
\begin{equation} \label{ass:f}
 f(s) \geq \frac{1}{C}|s|^{\max\{1,2[(2^\# - \delta)'-1]\}}-C \quad  \text{and} \quad  |f''(s)|\leq C(|s|^{2^*/2-1}+1)
 \end{equation}
for all $s\in \R$. Here $2^*$ is the dimension dependent Sobolev exponent: $2^* = \frac{2N}{N-2}$ if $N\geq 3$ and is any fixed constant greater than $2$ if $N\leq 2$. Likewise $2^\#$ is the optimal parameter for the trace embedding given by $2^\#:=\frac{2(N-1)}{N-2}$ if $N\geq3$ and any fixed constant greater than $2$ if $N\leq 2$. We also use the notation $p'$ to refer to the conjugate exponent of $p,$ which satisfies $\frac{1}{p} + \frac{1}{p'} = 1.$
Furthermore, we assume that $f$ is a double-well potential, and up to translation and scaling, we may assume
\begin{equation}\label{ass:f1}
 f(0) = f(1) = 0 \quad \text{and} \quad f(s)>0 \text{ for all }s\not\in \{0,1\}.
\end{equation}

\item
For the reaction rate function $R,$ we assume that there is $G\in C^1(\R^2)$ such that 
\begin{equation}\label{ass:G0}
\partial_w G(s,w) = R(s,w)
\end{equation} for all $s\in \R$ and $w\in \R.$ We suppose that the reaction rate is strictly decreasing in the second variable, in the sense that there is $C>0$ such that
\begin{equation}\label{ass:R1}
\left(R(s,w_2) - R(s,w_1) \right)(w_2-w_1) \leq -\frac{1}{C} |w_2-w_1|^2
\end{equation}
for all $s\in \R$ and $w_1,w_2 \in \R$. Further, for some $C>0$ and $\delta>0$ introduced above, the growth condition
\begin{equation}\label{ass:R2} |R(s,w)| \leq C(|s| + |w|^{2^\# - \delta-1} + 1)
\end{equation}
holds for all $s\in \R$ and $w \in \R$. We assume there is a constant $C>0$ such that the pointwise bound 
\begin{equation}\label{ass:R3}
|R(s,\pm 1)|\leq C
\end{equation}
is satisfied for any choice of $s \in \R.$ 

\end{enumerate}

Testing (\ref{ass:R1}) with $w_2=1$ or $-1$ and $w_1=0$, and using (\ref{ass:R3}) ,we find that $|R(s,0)|\leq C$ for some constant $C>0.$ 
From this bound and (\ref{ass:R1}) it follows there is $C>0$ such that, for all $s\in \R$ and $w\in \R$,
\begin{equation}\label{ass:R4}
-w R(s,w)\geq \frac{1}{C}|w|^2  -C.
\end{equation}
To derive growth conditions of the function $G,$ we may, without loss of generality, replace $G(s,w)$ by $G(s,w) -G(s,0)$ thereby enforcing $\partial_s G(s,0) = 0.$ Consequently, the fundamental theorem of calculus, (\ref{ass:R2}), and Young's inequality imply
\begin{equation}\label{ass:G1}
 |G(s,w)| \leq C(|s||w| + |w|^{2^\# - \delta} + 1) \leq C(|s|^{(2^\# - \delta)'} + |w|^{2^\# - \delta} + 1) 
\end{equation}
for all $s\in \R$ and $w\in \R.$

\begin{remark}
Existence of $G \in C^1(\R^2)$ satisfying  (\ref{ass:G0}) implies that we may without loss of generality assume
\begin{equation}\nonumber
G(s,w):=\int_{0}^w R(s,\rho) \, d\rho.
\end{equation}
\end{remark}
With this remark in hand, we can derive coercivity conditions for $-G$ as well:
\begin{equation}\label{ass:G2}
G(s,w) = \int_{0}^w R(s,\rho) \, d\rho\leq \int_0^w \left(-\frac{1}{C} |\rho| + R(s,0) \right) d \rho \leq -\frac{1}{C}|w|^2 + C|w| \leq -\frac{1}{C}|w|^2+C,
\end{equation}
where the computation was done for $w>0$, but the final inequality holds for all $w\in \R.$

\begin{remark}
We note that the coercivity hypothesis arising in (\ref{ass:f}) is needed to improve trace estimates for weak solutions of the CHR$_\epsilon$ model (\ref{pde:CHRnoElastic}) uniform in $\epsilon>0$. This estimate will interact with (\ref{ass:G1}), which appears in the metric of the gradient flow structure, to give an appropriate compactness on the time derivative of solutions of (\ref{pde:CHRnoElastic}) for the sharp interface limit.

In this vein, changing the coercivity (\ref{ass:f}) to have growth $\frac{2p}{2^{\#}-\delta-1}$ one can replace  $|s|$ by $|s|^p$ in (\ref{ass:R2}).
\end{remark}

\section{Functional Framework} \label{sec:funcFramework}

We introduce a framework, which naturally extends the functional framework based on the inverse Laplacian $(-\Delta)^{-1}:H^{-1}_{(0)} \to H^1_{(0)}$ used to express the Cahn--Hilliard equation to a setting appropriate for the CHR$_\epsilon$ model. We build off of a novel functional framework introduced by Kraus and Roggensack \cite{kraus2017} to find solutions of a viscous CHR model. In their exposition, duality revolves around $L^2(\Omega),$ which is the natural space for an Allen--Cahn type equation, but in the case of the CHR$_\epsilon$ model, the gradient flow structure is succinctly captured when the primary functions spaces are $H^1(\Omega)$ and its dual $H^1(\Omega)^*$. This approach and a coercivity estimate (Lemma \ref{lem:H1dualbdd} below) will allow us to prove existence of weak solutions to the CHR$_\epsilon$ model (\ref{pde:CHRnoElastic}) on domains with limited regularity assumptions and no viscosity (cf. \cite{kraus2017}) and will further capture the dissipation dynamics for the sharp interface model.

We introduce functionals to define the gradient structure. Let 
\begin{equation}\label{def:functA}
\begin{aligned}
\mathcal{A} & :L^{(2^\#-\delta)'}(\partial \Omega)\times H^1(\Omega) \to \R\cup\{\infty\}, \\
\mathcal{A}(c,v) & := \frac{1}{2}\int_\Omega \|\nabla v\|^{2} \, dx - \int_{\partial \Omega} G(c,v) \ d\mathcal{H}^{N-1}.
\end{aligned}
\end{equation}
This functional is proper, lower semi-continuous, and convex in the second input by assumptions (\ref{ass:G0}), (\ref{ass:R1}), and (\ref{ass:G1}). Furthermore, define 
\begin{equation}\nonumber
\begin{aligned}
\mathcal{B} & :L^{(2^\#-\delta)'}(\partial \Omega)\times H^1(\Omega) \to H^1(\Omega)^*, \\
\la \mathcal{B}(c,v), \xi \ra_{H^{1}(\Omega)^*,H^1(\Omega)}& :=  \int_\Omega \nabla v \cdot  \nabla \xi \ dx - \int_{\partial \Omega} R(c,v) \xi \ d \mathcal{H}^{N-1}.
\end{aligned}
\end{equation}

Under the monotonicity assumption (\ref{ass:R1}) and growth condition (\ref{ass:R2}) on $R$, we can show for fixed $c\in L^{(2^{\#}-\delta)'}(\partial \Omega)$ that $\mathcal{B}_c(\cdot) :=\mathcal{B}(c,\cdot)$ is strictly monotone, bounded, and coercive. Consequently, we may define (see Lemma 1 in \cite{kraus2017}) the bounded and continuous inverse operator
\begin{equation}\label{def:functBinv}
\begin{aligned}
\ov{\mathcal{B}} & :L^{(2^\#-\delta)'}(\partial \Omega)\times H^1(\Omega)^* \to H^1(\Omega), \\
\ov{\mathcal{B}}(c,v^*) & :=  \mathcal{B}^{-1}_c(v^*) .
\end{aligned}
\end{equation}
Setting $\mathcal{A}_c(\cdot):=\mathcal{A}(c,\cdot)$, using (\ref{ass:G0}) and (\ref{ass:R2}), given that the directional derivative in $H^1(\Omega)$ of $\mathcal{A}_c$ at $v\in H^1(\Omega)$ may be explicitly computed, we have
\begin{equation}\label{eqn:subdiffToBrelation}
v^*\in \partial \mathcal{A}_c(v)\subset H^1(\Omega)^* \iff v^*= \mathcal{B}(c,v).
\end{equation}
Hence, we can encapsulate equations $\partial_t c = \Delta \mu$ and $\partial_n \mu = R(c,\mu)$ of (\ref{pde:CHRnoElastic}) in the inclusion 
\begin{equation} \label{eqn:gradFlowStruct}
 -\partial_t c \in \partial \mathcal{A}_c(\mu),
\end{equation} thereby providing a gradient flow structure (if one is willing to temporarily forgo the definition of $\mu$). As this is a differential inclusion, there is a natural dual formulation. Let $\mathcal{A}^*$ be the Legendre--Fenchel transform of $\mathcal{A}$ with respect to the second input in the dual of $H^1(\Omega)$, that is,
\begin{equation}\label{def:functAconj}
\begin{aligned}
\mathcal{A}^* & :L^{(2^\#-\delta)'}(\partial \Omega)\times H^1(\Omega)^* \to \R\cup\{\infty\} ,\\
\mathcal{A}^*_c(v^*) & := (\mathcal{A}_c)^*(v^*) = \sup_{v\in H^1(\Omega)}\left\{\langle v^*,v \rangle_{H^1(\Omega)^*,H^1(\Omega)}-\mathcal{A}_c(v)\right\} .
\end{aligned}
\end{equation}
For $v \in H^1(\Omega)$ and $v^* \in H^1(\Omega)^*$, we will be interested in the Legendre--Fenchel inequality \cite{rockafellar-convex} $$\la v^*,v \ra_{H^1(\Omega)^*,H^1(\Omega)} \leq \mathcal{A}_c(v) + \mathcal{A}^*_c(v^*),$$ where particularly in the case of equality one has
\begin{equation}\label{eqn:legendreFenchelIneq}
\la v^*,v \ra_{H^1(\Omega)^*,H^1(\Omega)} = \mathcal{A}_c(v) + \mathcal{A}^*_c(v^*) \iff v^* \in \partial \mathcal{A}_c(v) \iff v \in \partial \mathcal{A}^*_c(v^*) .
\end{equation}
Making use of (\ref{def:functBinv}), (\ref{eqn:subdiffToBrelation}), and (\ref{eqn:legendreFenchelIneq}), we have
\begin{align}
\partial \mathcal{A}^*_c (v^*) = \{\ov{\mathcal{B}}(c,v^*)\}\subset H^1(\Omega), \label{eqn:subdiffH1}
\end{align} which, looking back to (\ref{eqn:gradFlowStruct}), provides a way to express $\mu$ in terms of $c$ and $\partial_t c$ via convex duality. Explicitly,
\begin{equation}\label{def:muViaB}
\mu = \ov{\mathcal{B}}(c,-\partial_t c).
\end{equation}
To tie this to the CHR$_\epsilon$ model (\ref{pde:CHRnoElastic}), a solution given in terms of the pair $(c,\mu)$ will satisfy the formal relation $\partial I_\epsilon [c] = \mu = \partial \mathcal{A}^*_c(-\partial_t c)$ for all time, and (\ref{def:muViaB}) shows that the second equality is in fact a definition for $\mu$. Consequently, to construct a solution of the CHR$_\epsilon$ model, one is first motivated to define $c$ such that $\partial I_\epsilon [c] =  \partial \mathcal{A}^*_c(-\partial_t c)$, thereby recovering (\ref{eqn:doublyNonlinearDiffInclusion}). This will be carried out in Theorem \ref{thm:CHRSolnexist}.

Finally, we conclude this section with an $H^1$-dual bound essential to our compactness arguments.
\begin{lem} \label{lem:H1dualbdd}
Let $\Omega \subset \R^N$ be a bounded, open set with Lipschitz boundary. Assume hypotheses (\ref{ass:f}) to (\ref{ass:R3}) hold with $\mathcal{A}^*$ defined as in (\ref{def:functAconj}). Suppose that $\|c\|_{L^{(2^{\#}-\delta)'}(\partial \Omega)}\leq \alpha.$ Then there exists $C_\alpha>0$ such that 
\begin{equation}\nonumber
\mathcal{A}^*_c(v^*)\geq \frac{1}{C_\alpha}\|v^*\|_{H^{1}(\Omega)^*}^{(2^\#-\delta)'} -C_\alpha. 
\end{equation}
\end{lem}
\begin{proof}
Under the growth assumption (\ref{ass:G1}) on $G$ and the definition of $\mathcal{A}$ in (\ref{def:functA}), for $v\in H^1(\Omega),$ we have 
\begin{align*}
\mathcal{A}_c(v)\leq & C\big(\|v\|_{H^1(\Omega)}^2 +\|v\|_{L^{2^\#-\delta}(\partial \Omega)}^{2^\#-\delta}+\|c\|_{L^{(2^{\#}-\delta)'}(\partial \Omega)}^{(2^{\#}-\delta)'} + 1\big) \\
\leq & C_\alpha(\|v\|_{H^1(\Omega)}^{2^\#-\delta}+1)
\end{align*}
where in the second inequality we have used the trace inequality $\|v\|_{L^{2^\#-\delta}(\partial \Omega)}\leq C\|v\|_{H^1(\Omega)}$ and the assumption on $c$ in the lemma. We compute by definition of the conjugate
\begin{equation}\label{eqn:Aconjsup}
\begin{aligned}
\mathcal{A}_c^*(v^*) = & \sup_{v\in H^1(\Omega)} \Big\{\la v^*,v\ra_{H^1(\Omega)^*,H^1(\Omega)} - \mathcal{A}_c(v) \Big\} \\
\geq & \sup_{v\in H^1(\Omega)} \Big\{\la v^*,v\ra_{H^1(\Omega)^*,H^1(\Omega)} - C_\alpha\|v\|_{H^1(\Omega)}^{2^\#-\delta}\Big\}-C_\alpha.
\end{aligned}
\end{equation}
Since $2^\#-\delta>1,$ there is a unique maximizer $v_0$ to the latter supremum, and computing the Gateaux derivative, it must satisfy the following relation for all $\xi\in H^1(\Omega):$
$$\la v^*, \xi\ra_{H^1(\Omega)^*,H^1(\Omega)} = (2^\#-\delta)C_\alpha\|v_0\|^{2^\#-\delta-2}_{H^1(\Omega)}(v_0,\xi)_{H^1(\Omega)}.$$
Furthermore as the right hand side is maximized over the $H^1(\Omega)$ unit ball at $ \xi =v_0 / \|v_0\|_{H^1(\Omega)}$, this implies $$\frac{1}{\|v_0\|_{H^1(\Omega)}}\la v^*, v_0\ra_{H^1(\Omega)^*,H^1(\Omega)} = \|v^*\|_{H^{1}(\Omega)^*} =  (2^\#-\delta)C_\alpha\|v_0\|^{2^\#-\delta-1}_{H^1(\Omega)}.$$
Consequently, evaluating the supremum in (\ref{eqn:Aconjsup}) at its maximizer, we find
\begin{align*}
\mathcal{A}_c^*(v^*) \geq & \la v^*,v_0 \ra_{H^1(\Omega)^*,H^1(\Omega)} - C_\alpha \|v_0\|_{H^1(\Omega)}^{2^\#-\delta}-C_\alpha \\
= & (2^\#-\delta-1)C_\alpha\|v_0\|^{2^\#-\delta}_{H^1(\Omega)} -C_\alpha \\
= & \frac{ (2^\#-\delta-1)C_\alpha}{((2^\#-\delta)C_\alpha)^{(2^\#-\delta)'}}\|v^*\|_{H^{1}(\Omega)^*}^{(2^\#-\delta)'} -C_\alpha.
\end{align*}
Up to redefinition of $ C_\alpha$, this completes the lemma.
\end{proof}

\section{Existence of Solutions for the CHR$_\epsilon$ Model}\label{sec:CHRexist}

To characterize the limiting behavior of solutions to the CHR$_\epsilon$ model (\ref{pde:CHRnoElastic}), we introduce a solution concept at the fixed $\epsilon >0$ level in which the optimal dissipation inequality plays a fundamental role. Subsequently, a proposition allows us to connect this weak solution concept to a more standard distributional formulation. Finally, we prove existence of solutions in Lipschitz domains using a minimizing movements scheme.

\begin{definition}\label{def:CHRweakSoln}
Let $\Omega \subset \R^N$ be an open bounded set with Lipschitz boundary. We say that $c$ is a \emph{weak solution of the CHR$_\epsilon$ model} (\ref{pde:CHRnoElastic}) on $\Omega \times (0,T)$ if for $\delta > 0$ as in (\ref{ass:f})
\begin{align*}
& c\in   L^{\infty}(0,T;H^1(\Omega))\cap C([0,T);L^2(\Omega)),\\
& \partial_t c\in L^{(2^\#-\delta)'}(0,T;H^{1}(\Omega)^*), \\
& c(0) = c_0 \in H^1(\Omega),
 \end{align*} and there exists $\mu \in L^{2}(0,T;H^1(\Omega))$ such that the sharp energy dissipation inequality
 \begin{equation}\label{def:CHR-dissipation}
\begin{aligned} 
I_\epsilon[c(T_*)] +\int_0^{T_*} \mathcal{A}_c(\mu) + \mathcal{A}^*_c(-\partial_t c)  \, dt \leq I_\epsilon[c_0]
\end{aligned}
\end{equation}
holds for almost every $T_*$ in $(0,T)$, where $\mu = \mu(x,t)$ is given by
\begin{equation}\label{def:muCHR}
\int_\Omega \mu(t)\xi \, dx =  \int_\Omega \left( \epsilon \nabla c(t)\cdot \nabla \xi +\frac{1}{\epsilon}f'(c(t))    \xi \right) dx,
\end{equation}
for all $\xi \in H^1(\Omega)$ and almost every $t$ in $(0,T)$.
\end{definition}

We remark that (\ref{def:muCHR}) encapsulates the relation $\mu = -\epsilon\Delta c +\frac{1}{\epsilon}f'(c) =  \partial I_\epsilon [c]$ and $\partial_n c = 0$, which in conjunction with (\ref{def:CHR-dissipation}) provides the relation $\mu \in \partial\mathcal{A}^*_c(-\partial_t c),$ as in (\ref{eqn:gradFlowStruct}) and (\ref{def:muViaB}). This connection is discussed in Remark \ref{rmk:epsilonConsistency} below. Thereby (\ref{def:CHR-dissipation}) and (\ref{def:muCHR}) compactly represent the CHR$_\epsilon$ model (\ref{pde:CHRnoElastic}).

In the context of the Cahn--Hilliard and Mullins--Serkerka equations, there are a variety of ways that $H^1(\Omega)$ is considered in relation to its dual $H^1(\Omega)^*$. The chemical potential $\mu\in H^1(\Omega)$ is related to $\partial_t c \in H^1(\Omega)^*$ through the Neumann Laplacian, while  $\partial_t c \in H^1(\Omega)^*$ is realized as the derivative of $c$ through the evolution triple $H^1(\Omega)\hookrightarrow L^2(\Omega) \hookrightarrow H^1(\Omega)^*.$  As such, the setting briefly discussed by Ambrosio, Gigli, and Savar\'e \cite[Section 1.4]{AGS-GradFlows}, which recovers the differential inclusion from the dissipation, does not directly apply, as $c$ is not an absolutely continuous curve in $H^1(\Omega)$. However, this is rectified in our situation via the following proposition. 
\begin{prop}\label{prop:energyDiff}
Let $\Omega \subset \R^N$ be a bounded, open set with $C^2$ boundary. Assume hypothesis (\ref{ass:f}) holds. If $c\in L^\infty(0,T;H^1(\Omega))\cap L^2(0,T;H^2(\Omega)) \cap H^1(0,T;H^1(\Omega)^*)$ with $\mu \in L^2(0,T;H^1(\Omega))$ satisfying (\ref{def:muCHR}), then $t\mapsto I_\epsilon[c(t)]$ is absolutely continuous on $[0,T]$ and for almost every $t$ in $(0,T)$
$$\frac{d}{dt}I_\epsilon[c(t)] = \la\partial_t c(t), \mu(t) \ra_{H^1(\Omega)^*,H^1(\Omega)}.$$
\end{prop}
\begin{proof}
For notational ease, we let $\epsilon = 1$. Extending $c$ by reflection, we may assume $(0,T)$ is replaced by $\R$. For $\phi_\eta:\R \to \R$ a mollifier with radius $\eta>0$, we can mollify in time to find $c_\eta : =  c* \phi_\eta \in C^\infty (\R; H^3(\Omega))$, which converges to $c$ in the space $L^2(0,T;H^1(\Omega))\cap H^1(0,T;H^1(\Omega)^*)$. As $f'(c) \in L^2(0,T;H^1(\Omega))$ by $c$'s regularity assumptions, $\mu_\eta \in L^2(0,T;H^1(\Omega))$ may be defined as in (\ref{def:muCHR}) and can explicitly be written as $\mu_\eta = (\mu - f'(c))*\phi_\eta +f'(c_\eta)$. We have regularized $c_\eta$ only in time, which allows us to retain the relation (\ref{def:muCHR}) between $c_\eta$ and $\mu_\eta$. 

We will show that 
$$\frac{d}{dt}I_1[c_\eta(t)] = \la\partial_t c_\eta(t), \mu_\eta(t) \ra_{H^1(\Omega)^*,H^1(\Omega)}.$$
First for $\tau > 0$, we use Taylor's theorem to rewrite the difference quotient for the nonlinearity $f(c_\eta)$ in terms of $c_\eta$ as
\begin{equation}\label{eqn:taylorfc}
\begin{aligned}
\left|\frac{f(c_\eta (t+\tau)) - f(c_\eta(t))}{\tau}- f'(c_\eta (t))\frac{c_\eta (t+\tau) - c_\eta (t)}{\tau} \right| \leq \frac{1}{2}\|f''\|_{L^\infty(\mathcal{I}_{t,\eta,\tau})}\tau\left(\frac{c_\eta (t+\tau) - c_\eta (t)}{\tau}\right)^2,
\end{aligned}
\end{equation}
where $\mathcal{I}_{t,\eta,\tau}$ is the interval with endpoints $c_\eta(t)$ and $c_\eta(t+\tau)$. Note that like the values $c_\eta$, the length of the interval varies with $x\in \Omega.$
Taking the absolute value and integrating in space, we bound the right-hand side of the above estimate using (\ref{ass:f}), $c_\eta \in L^\infty(0,T;H^1(\Omega)) \hookrightarrow L^\infty(0,T;L^{2^*}(\Omega))$, H\"older's inequality with $p = \frac{2^*2}{2^*-2}$ and $q = \frac{2^*2}{2^*+2}$, the fundamental theorem of calculus, Jensen's inequality, and $\frac{4}{2^*+2}\leq 1$ by
\begin{equation}\label{eqn:taylorfcErr}
\begin{aligned}
&\frac{\tau}{2}\int_\Omega \|f''\|_{L^\infty(\mathcal{I}_{t,\eta,\tau})}\left(\frac{c_\eta (t+\tau) - c_\eta (t)}{\tau}\right)^2 \, dx  \\
& \leq \tau C\left(\|c_\eta\|_{L^\infty(0,T;H^1(\Omega))}^{\frac{2^*  -2}{2}} +1\right) \left(\int_\Omega \left(\intbar_0^\tau \partial_t c_\eta ( t + s) \, ds \right)^{2^*\frac{4}{2^*+2}} dx\right)^{1/q} \\
& \leq \tau^{1-\frac{1}{q}} C\left(\|c_\eta\|_{L^\infty(0,T;H^1(\Omega))}^{\frac{2^*  -2}{2}} +1\right) \left(\left(\int_\Omega \int_0^T (\partial_t c_\eta)^{2^*} \, dt \, dx\right)^{1/q}+1\right) = \tau^{1-\frac{1}{q}} C(c,\eta).
\end{aligned}
\end{equation}

We apply estimates (\ref{eqn:taylorfc}) and (\ref{eqn:taylorfcErr}) centered at $t$ and $t+\tau$ and use the relation (\ref{def:muCHR}) in the following computation:
\begin{equation}\nonumber
\begin{aligned}
\frac{d}{dt}I_\epsilon[c_\eta(t)] = &\lim_{\tau \to 0} \int_\Omega \frac{f(c_\eta(t+\tau) -f(c_\eta(t))}{\tau} +\frac{\|\nabla c_\eta(t+\tau)\|^2 - \|\nabla c_\eta(t)\|^2}{2\tau} \, dx \\
 = &\lim_{\tau \to 0} \int_\Omega \frac{1}{2} \left(f'(c_\eta(t+\tau)) \frac{c_\eta(t+\tau) - c_\eta(t)}{\tau}+ f'(c_\eta(t)) \frac{c_\eta(t+\tau) - c_\eta(t)}{\tau}\right. \\
 &\left. \quad +\frac{(\nabla c_\eta (t+\tau) - \nabla c_\eta(t),\nabla c_\eta(t+\tau))}{\tau} + \frac{(\nabla c_\eta (t+\tau) - \nabla c_\eta(t),\nabla c_\eta(t))}{\tau} \right) dx  \\
 = & \lim_{\tau \to 0}  \int_\Omega \frac{1}{2}\left(\mu_\eta(t+\tau) + \mu_\eta(t) \right) \frac{c_\eta(t+\tau) - c_\eta(t)}{\tau} \, dx \\
= & \la\partial_t c_\eta(t), \mu_\eta(t) \ra_{H^1(\Omega)^*,H^1(\Omega)}.
\end{aligned}
\end{equation}
Consequently, for $0<s<t<T$, by the fundamental theorem of calculus,
$$I_1 [c_\eta (s)] + \int_{s}^t  \la\partial_t c_\eta(\tau), \mu_\eta(\tau) \ra_{H^1(\Omega)^*,H^1(\Omega)}\, d\tau=  I_1 [c_\eta (t)]. $$
Passing to the limit as $\eta \to 0,$ we find 
\begin{equation}\nonumber
 I_1 [c (s)] + \int_{s}^t  \la\partial_t c(\tau), \mu(\tau) \ra_{H^1(\Omega)^*,H^1(\Omega)}\,d\tau=  I_1 [c (t)],
 \end{equation}
for almost all $0<s<t<T$, which concludes the theorem.
\end{proof}

\begin{remark}\label{rmk:epsilonConsistency}
We note that if $c$ is a solution of the CHR$_\epsilon$ model in the sense of Definition \ref{def:CHRweakSoln} satisfying the hypothesis of Proposition \ref{prop:energyDiff} (which given the solution definition only additionally imposes domain regularity and that $\partial_t c \in L^2(0,T;H^1(\Omega)^*)$), then (\ref{def:CHR-dissipation}) implies
$$ \mathcal{A}_{c(t)}(\mu(t)) + \mathcal{A}^*_{c(t)}(-\partial_t c(t)) \leq \la - \partial_t c(t), \mu(t) \ra_{H^1(\Omega)^*,H^1(\Omega)}$$ for almost every $t\in (0,T).$
This inequality shows us that there is equality in the case of the Legendre--Fenchel inequality for $\mathcal{A}_c(\cdot)$ (see (\ref{eqn:legendreFenchelIneq})) and consequently $-\partial_t c(t) \in \partial \mathcal{A}_{c(t)}(\mu(t)),$ which by (\ref{eqn:subdiffToBrelation}) may be reinterpreted as
 \begin{equation}\label{def:CHR-ELeqn}
\begin{aligned} 
-\la \partial_t c(t) , \xi \ra_{H^{1}(\Omega)^*,H^1(\Omega)} = & \int_\Omega \nabla \mu(t)\cdot \nabla \xi \ dx- \int_\Gamma R(c(t),\mu(t)) \xi \, d \mathcal{H}^{N-1} \quad \text{ for all }\xi\in H^1(\Omega),
\end{aligned}
\end{equation}
which is perhaps the expected weak formulation in terms of distributional derivatives. 

Further, from a result of the second author's thesis \cite{stinson-phd}, for a sufficiently regular domain and with $2^\#-\delta = 2$, there exists $c \in L^2(0,T;H^3(\Omega))\cap H^1(0,T;H^1(\Omega)^*)$ such that $c(t) \to c_0$ in $H^1(\Omega)$ as $t\downarrow 0$ and both (\ref{def:muCHR}) and (\ref{def:CHR-ELeqn}) are satisfied. By (\ref{eqn:subdiffToBrelation}) and (\ref{def:CHR-ELeqn}), we have $-\partial_t c(t) \in \partial \mathcal{A}_c(\mu(t))$ for almost every $t$ in $(0,T)$. Then using Proposition \ref{prop:energyDiff} and equality in the Legendre--Fenchel inequality (\ref{eqn:legendreFenchelIneq}), we have
$$\frac{d}{dt}I_\epsilon [c](t) + \mathcal{A}_{c(t)}(\mu(t)) + \mathcal{A}^*_{c(t)}(-\partial_t c(t)) = 0 $$
for almost every $t$ in $(0,T).$ Integrating the above equality on $(s,T_*)$, and then letting $s\downarrow 0$ implies that the optimal energy dissipation relation (\ref{def:CHR-dissipation}) is also satisfied. Consequently under mild regularity assumptions, a weak solution with distributional formulation as in (\ref{def:CHR-ELeqn}) is equivalently characterized by the optimal dissipation relation as in Definition \ref{def:CHRweakSoln}.
\end{remark}

To account for the case when $2^\# - \delta > 2$, i.e., $R$ is superlinear, we will turn directly to the minimizing movements scheme to recover a solution in the sense of Definition \ref{def:CHRweakSoln}. For the reader's convenience, we outline the essential elements of this argument, for which the general strategy is contained in \cite[Chapter 3]{AGS-GradFlows}.

\begin{thm}\label{thm:CHRSolnexist}
Let $\Omega \subset \R^N$ be a bounded, open set with Lipschitz boundary. Assume hypotheses (\ref{ass:f}) to (\ref{ass:R3}) hold. For $T>0$, $\epsilon > 0$, and $c_0\in H^1(\Omega)$, there exists a solution of the CHR model (\ref{pde:CHRnoElastic}) in the sense of Definition \ref{def:CHRweakSoln}. 
\end{thm}

\begin{proof}
\textbf{Step 1: Minimizing movements scheme.} For $\tau =\tau_m := T/m>0$,  with $m\in \mathbb{N},$ we define the following iterative scheme. Let $c_\tau^0 = c_0.$ Define 
\begin{equation} \label{minmovscheme}
c_\tau^i := \operatorname*{argmin}\limits_{c\in H^1(\Omega)} \left\{ I_\epsilon[c] + \tau \mathcal{A}_{c^{i-1}_\tau}^*\left(-\frac{c - c_\tau^{i-1}}{\tau}\right) \right\}, 
\end{equation}
where $I_\epsilon$ and $\mathcal{A}^*$ are defined in (\ref{def:energy}) and (\ref{def:functAconj}), respectively. As in \cite[Lemma 7]{kraus2017}, by the direct method, minimizers exist.

Define $c_\tau^-$ the piecewise constant function
\begin{equation}\nonumber
c_\tau^-(t) : = 
c^i_\tau \quad \text{if } t\in[i\tau,(i+1)\tau),
\end{equation} and $\hat c_\tau$ the linear interpolant
\begin{equation}\nonumber
\hat c_\tau(t) : = 
\frac{(i+1)\tau-t}{\tau}c^i_\tau+\frac{t - i\tau}{\tau}c^{i+1}_\tau \quad \text{if } t\in[i\tau,(i+1)\tau).
\end{equation} 
We now introduce the variational interpolant $\tilde c_\tau$ given by $\tilde c_\tau (i\tau ) := c_\tau^i$ and for all other $t$
\begin{equation} \label{minmovscheme:VarInterp}
\tilde c_\tau(t) := \text{argmin}_{c\in H^1(\Omega)} \left\{ I_\epsilon[c] + (t-\tau\lfloor t/\tau\rfloor) \mathcal{A}_{c^{-}_\tau}^*\left(-\frac{c - c_\tau^-(t)}{t-\tau\lfloor t/\tau\rfloor }\right) \right\}.
\end{equation}
Motivated by (\ref{def:muViaB}), we use a difference quotient for the time derivative and define the interpolating chemical potential $\tilde \mu_\tau$ by 
\begin{equation}\label{def:discChemPot}
\tilde \mu_\tau (t) : = \ov{\mathcal{B}}\left(c^{-}_\tau(t), -\frac{\tilde c_\tau(t) - c^{-}_\tau(t)}{t-\tau\lfloor t/\tau\rfloor}\right) \quad \text{ for } t\neq \tau\lfloor t/\tau\rfloor.
\end{equation}

\textbf{Step 2: Optimal dissipation for approximation.} We show that for $i\in \N_0$
\begin{equation}\label{eqn:discDissRelation}
I_\epsilon[c_\tau^{i+1}] + \int_{i\tau}^{(i+1)\tau} \left(\mathcal{A}_{c_\tau^-} (\tilde \mu_\tau) + \mathcal{A}_{c_\tau^-}^* (-\partial_t \hat c_\tau) \right) dt \leq I_\epsilon[c_\tau^i].
\end{equation}
For simplicity, we drop the subscript $\tau$ and, without loss of generality, suppose that $i = 0.$ Define the function 
\begin{equation}\nonumber
F(t) : =I_\epsilon [\tilde c(t)]  +t \mathcal{A}^*_{c_0}\left(-\frac{\tilde c(t) - c_0}{t}\right).
\end{equation}
We \textbf{claim} that $F(t)$ is continuous on $(0,\tau]$ and that for all $t\in (0,\tau)$
\begin{equation}\label{eqn:limsupDeriv}
\partial_+ F(t) : =\limsup_{s\downarrow t}\frac{F(s) - F(t)}{s-t} \leq -\mathcal{A}_{c_0}(\tilde \mu (t)).
\end{equation}

Supposing momentarily that we have proven the claim, we show how to obtain (\ref{eqn:discDissRelation}). By Lemma \ref{lem:H1dualbdd}, $\mathcal{A}_{c_0}(\tilde \mu (t))> -C$ for some $C>0$ for all $t$. Using this bound in (\ref{eqn:limsupDeriv}), we have that $\partial_+(F(t) -Ct) < 0$ for all $t \in (0,\tau)$. To see that $t \mapsto F(t) - Ct$ is non-increasing suppose to the contrary that there are $t$ and $s$ in $(0,\tau)$ such that $t<s$ and $$F(t) - Ct< F(s) - Cs,$$ and define $$t_0 : = \inf\{t':t<t', \   F(t) - Ct< F(t') - Ct'\}.$$ By continuity, $F(t_0) - Ct_0 = F(t) - Ct $, and $\partial_+(F(t_0) -Ct_0) \geq 0$, a contradiction confirming that the function $F(t) - Ct$ is non-increasing.
From this, it follows that $F$ belongs to $BV_{loc}((0,\tau])$ with $D_a F = \partial_+ F \mathcal{L}^1\llcorner(0,\tau)$ and the singular part of its derivative is non-increasing, i.e., $\frac{dD_s F}{d|D_s F|} \leq 0$. We integrate and use the relation (\ref{eqn:limsupDeriv}) on $(s,\tau)$ to find
\begin{equation}\label{eqn:tempDiss}
F(\tau) -F(s) = \int_s^\tau dDF \leq \int_s^\tau -\mathcal{A}_{c_0}(\tilde \mu (t))\, dt.
\end{equation} By the definition of $F(s),$ we have that
$F(s) \leq I_\epsilon[c_0] + s\mathcal{A}^*_{c_0}(0)$. Using this inequality in (\ref{eqn:tempDiss}), we recover (\ref{eqn:discDissRelation}) by letting $s\to 0.$

We now prove the claim. First, to see that $F$ is continuous on $(0,\tau]$ (it is always finite), suppose that there is $t_i \to t$ such that $\lim_{i\to \infty} F(t_i) = \alpha \neq F(t).$ If $\alpha <F(t)$, then up to a subsequence, there is $\bar c \in H^1(\Omega)$ such that $\tilde c(t_i)\wkto \bar c$ weakly in $H^1(\Omega).$ Up to a further subsequence, $\frac{\tilde c(t_i) - c_0}{t_i} \to \frac{\bar c - c_0}{t}$ strongly in $H^1(\Omega)^*$. Applying lower semi-continuity of $I_\epsilon$ and continuity of $\mathcal{A}^*$, we have
\begin{equation}\nonumber
I_\epsilon[\bar c] + t \mathcal{A}^*_{c_0} \left(-\frac{\bar c - c_0}{t}\right) \leq \alpha < F(t),
\end{equation}
a contradiction.
Similarly, if $\alpha > F(t)$, we can contradict the definition of $F(t_i)$ by using $\tilde c(t)$ as a competitor for $t_i$ sufficiently close to $t$. This concludes the proof of continuity.

To obtain the bound (\ref{eqn:limsupDeriv}), let $t<s$ and note by optimality of $\tilde c(s)$ at $s,$ one has
\begin{equation}\nonumber
F(s)\leq F(t) + s\mathcal{A}^*_{c_0} \left(-\frac{\tilde c(t) - c_0}{s}\right) -t\mathcal{A}^*_{c_0} \left(-\frac{\tilde c(t) - c_0}{t}\right).
\end{equation}
Rearranging the above inequality and dividing by $s-t$, we find (\ref{eqn:limsupDeriv}) using (\ref{eqn:legendreFenchelIneq}), (\ref{def:discChemPot}), and (\ref{eqn:subdiffH1}) as follows:
\begin{equation}\nonumber
\begin{aligned}
\limsup_{s\downarrow t}\frac{F(s) - F(t)}{s-t} \leq & \,\frac{d}{ds}\Big|_{s=t} \left[s\mathcal{A}^*_{c_0} \left(-\frac{\tilde c(t) - c_0}{s}\right)\right] \\
= & \, \mathcal{A}^*_{c_0} \left(-\frac{\tilde c(t) - c_0}{t}\right) + t\left\la\partial  \mathcal{A}^*_{c_0} \left(-\frac{\tilde c(t) - c_0}{t}\right), \frac{\tilde c(t) - c_0}{t^2}\right\ra \\
= & \, \mathcal{A}^*_{c_0} \left(-\frac{\tilde c(t) - c_0}{t}\right) - \left\la  -\frac{\tilde c(t) - c_0}{t},  \tilde \mu (t)\right\ra_{H^1(\Omega)^*,H^1(\Omega)}  = -\mathcal{A}_{c_0}(\tilde \mu (t)).
\end{aligned}
\end{equation}

\textbf{Step 3: Compactness.} Telescoping, the dissipation relation (\ref{eqn:discDissRelation}) implies that for $i \in \N_0$
\begin{equation}\label{eqn:discDissLong}
I_\epsilon[\tilde c(i\tau)] +\int_0^{i\tau}  \left(\mathcal{A}_{c_\tau^-} (\tilde \mu_\tau) + \mathcal{A}_{c_\tau^-}^* (-\partial_t \hat c_\tau) \right) dt \leq I_\epsilon[c_0].
\end{equation}
Further, taking outer variations at the minimum $\tilde c_\tau$ from (\ref{minmovscheme:VarInterp}) in the direction of $\xi \in H^1(\Omega)$ identified as an element of $H^1(\Omega)^*$ via the triple $\xi \in H^1(\Omega)\hookrightarrow L^2(\Omega) \hookrightarrow H^1(\Omega)^*$ and using the definition of $\tilde \mu_\tau$ in (\ref{def:discChemPot}), for almost every $t\in (0,T)$, we have
\begin{equation}\label{eqn:discChemPotEL}
(\tilde \mu_\tau(t),\xi)_{L^2(\Omega)} =  \int_\Omega \big( \nabla \tilde c_\tau(t)\cdot \nabla \xi +f'(\tilde c_\tau(t))    \xi \big)\, dx.
\end{equation}
To obtain compactness for passing to the limit in these equations, we directly estimate norms. 
First, from (\ref{eqn:discDissLong}), (\ref{ass:G2}), and Lemma \ref{lem:H1dualbdd}, we have that 
\begin{equation}\label{eqn:energyBounds}
\begin{aligned}
\|\tilde c_\tau \|_{L^\infty(0,T;H^1(\Omega))} + \| c_\tau^- \|_{L^\infty(0,T;H^1(\Omega))} +\|\hat c_\tau \|_{L^\infty(0,T;H^1(\Omega))} & \leq C, \\
\|\tilde \mu_\tau\|_{L^2(0,T;H^1(\Omega))} + \|\partial_t \hat c_\tau\|_{L^{(2^\# -\delta)'}(0,T;H^1(\Omega)^*)} & \leq  C
\end{aligned}
\end{equation}
for some $C>0.$

Using the second bound from (\ref{eqn:energyBounds}), one can show by Jensen's inequality that
\begin{equation}\label{eqn:almostALShyp}
\int_0^{T-\delta} \|c_\tau^-(t+\delta) - c_\tau^- (t)\|^{(2^\# - \delta)'}_{H^1(\Omega)^*} \, dt \to 0 \text{ uniformly in }\tau \text{ as }\delta \to 0.
\end{equation}
From the definition of $\tilde c_\tau$ in (\ref{minmovscheme:VarInterp}) and Lemma \ref{lem:H1dualbdd}, one has that (\ref{eqn:almostALShyp}) holds for $\tilde c_\tau$ too. Consequently, by the Aubin--Lions--Simon compactness theorem \cite{simonCompact}, we have that $c_\tau^-, \ \hat c_\tau,$ and $\tilde c_\tau$ converge to $c \in L^\infty (0,T;H^1(\Omega))\cap W^{1,(2^\#-\delta)'}(0,T;H^1(\Omega)^*)$ in $L^{(2^\#-\delta)'}(0,T;L^2(\Omega))$. Note it is not a priori clear that the limits coincide, but this directly checked as in \cite[Theorem 1]{henselStinson-weakSolnMS} or \cite[Lemma 4.2.7]{stinson-phd}. Further, if one keeps track of the traces, using the triple $$c \in H^1(\Omega) \mapsto (c,{\rm Tr}(c)) \in L^2(\Omega) \times L^{(2^\# -\delta)'}(\partial \Omega) \mapsto c\in L^2(\Omega)\subset  H^1(\Omega)^*$$ as in \cite[Lemma 4.2.7]{stinson-phd}, we also find that ${\rm Tr}(c_\tau^-) \in L^\infty (0,T ; L^{(2^\# -\delta)'}(\partial \Omega))$ converges to ${\rm Tr}(c)$ strongly in $L^2(0,T;L^{(2^\# -\delta)'}(\partial \Omega))$ strongly. Lastly, it is clear that $\tilde \mu_\tau$ converges weakly in $L^2(0,T;H^1(\Omega))$ to some $\mu.$

\textbf{Step 4: Passing to the limit.} To pass to the limit in (\ref{eqn:discChemPotEL}) and recover (\ref{def:CHR-ELeqn}), one can integrate in time, pass to the limit $\tau \downarrow 0$, and then localize in time. To pass to the limit in the dissipation (\ref{eqn:discDissLong}), one can apply the argument of Theorem \ref{thm:BVsolnExist} to see that 
$$\int_0^s\left(\mathcal{A}_{c} ( \mu) + \mathcal{A}_{c}^* (-\partial_t  c) \right) dt\leq \liminf_{\tau \to 0} \int_0^s \left(\mathcal{A}_{c_\tau^-} (\tilde \mu_\tau) + \mathcal{A}_{c_\tau^-}^* (-\partial_t \hat c_\tau) \right) dt.$$ To obtain the limiting dissipation inequality, fix $s <T^*$ and suppose $\tau<T_*-s$. Then using Lemma \ref{lem:H1dualbdd} and (\ref{ass:G2}) to control the error for increasing the domain of integration from $(0,s)$ to $(0,\tau \lfloor T_*/\tau \rfloor)$ and applying (\ref{eqn:discDissLong}), we have
\begin{equation}\nonumber
\begin{aligned}
& I_\epsilon [c_\tau^- (T_*)] + \int_0^s \left(\mathcal{A}_{c_\tau^-} (\tilde \mu_\tau) + \mathcal{A}_{c_\tau^-}^* (-\partial_t \hat c_\tau) \right) dt  \\
& \leq I_\epsilon [c_\tau^- (\tau \lfloor T_*/\tau \rfloor)] + \int_0^{\tau \lfloor T_*/\tau \rfloor} \left(\mathcal{A}_{c_\tau^-} (\tilde \mu_\tau) + \mathcal{A}_{c_\tau^-}^* (-\partial_t \hat c_\tau) \right) dt + C(T_* -s)  \leq I_\epsilon[c_0]+ C(T_* -s).
\end{aligned}
\end{equation}
Noting that for almost every $T_*$ in $(0,T)$ we have $I_\epsilon [c(T_*)]\leq \liminf_{\tau\to 0} I_\epsilon [c_\tau^-(T_*)]$, one can directly apply lower semi-continuity to the above display and then let $s\uparrow T_*$ to conclude (\ref{def:CHR-dissipation}).
\end{proof}

\begin{remark}
We note that to obtain the optimal dissipation relation with minimal requirements on the domain geometry, it appears necessary to use De Giorgi's variational interpolation \cite{AGS-GradFlows} as was done in the previous proof. For comparison, in the case of a weak solution concept for a Cahn--Hilliard type equation without optimal dissipation \cite{garcke-CHdynamicBC}, there is no need to turn to variational interpolation. In the case of the viscous CHR model, the regularizing viscosity term gives that $\partial_t \hat c_\tau \in L^2(0,T;L^2(\Omega))$ allowing for an alternative argument for the optimal dissipation \cite{kraus2017}.
\end{remark}

\section{Convergence to the Sharp Interface Model}\label{sec:sharpInterface}

In this section, we relate the CHR$_\epsilon$ model (\ref{pde:CHRnoElastic}) to a free boundary model, whereby the underlying phenomena of phase coarsening is explicitly witnessed as a geometric interface evolution. 
The work of Meca, et al. \cite{meca_munch_wagner_2018} developed formal asymptotic expansions for the related viscous CHR$_\epsilon$ model, which indicates that solutions of the CHR$_\epsilon$ model (\ref{pde:CHRnoElastic}) should converge to solutions of the Mullins--Sekerka reaction model, whose strong formulation we introduce below. Our goal is to make this connection rigorous.

\begin{definition}\label{def:strongSolnMSR}
\begin{subequations}\label{eqn:strongMSR}
Let $N \geq 2$ and let $\Omega \subset \R^N$ be a bounded domain with $C^2$ boundary~$\p\Omega$. For $T>0$, 
let $\mathcal{S}=(\mathcal{S}(t))_{t \in [0,T)}$
be a time-dependent family of open subsets $\mathcal{S}(t) \subset \Omega$ such that $\Sigma(t) : = \partial \mathcal{S}(t) \cap \Omega$ is a smoothly evolving surface with inward-pointing normal $n_{\Sigma (t)}$.
For every $t \in [0,T)$, we denote
by $V_{\Sigma(t)}$ and $H_{\Sigma(t)}$ the associated normal velocity vector
and mean curvature vector, respectively. The family~$\mathcal{S}$ is said to be a \emph{strong solution of
the Mullins--Sekerka reaction (MSR) model} if for each $t \in (0,T)$ there exists
a chemical potential $\mu(\cdot,t)$ so that
\begin{align}
\label{eqn:MSR1}
\Delta \mu(\cdot,t) &= 0
&& \text{in } \Omega \setminus \Sigma(t),
\\ \label{eqn:MSR2}
V_{\Sigma(t)} &= - \jump{\partial_{n_{\Sigma(t)}} \mu } n_{\Sigma(t)}
&& \text{on } \Sigma(t),
\\ \label{eqn:MSR3}
\sigma H_{\Sigma(t)} &= \mu(\cdot,t) n_{\Sigma(t)}
&& \text{on } \Sigma(t),
\\ \label{eqn:MSR4}
\jump{\mu(\cdot,t)} &= 0
&&\text{on } \Sigma(t),
\\ \label{eqn:MSR5}
\partial_{n_{\p\Omega}} \mu(\cdot,t) &= R(\chi_{\overline{\mathcal{S}}(t)},\mu(\cdot,t))
&&\text{on } \p\Omega\setminus \partial \Sigma(t),
\\ \label{eqn:MSR6}
\partial \Omega \measuredangle \Sigma &= 90^\circ
&&\text{on } \p\Omega\cap \partial \Sigma(t),
\end{align}
where $\chi_{\overline{S}}$ is the characteristic function of the set $\overline{S}$ and the jump is oriented in the direction of the normal, i.e., such that $\jump{\partial_{n_{\Sigma}} \mu } = (\nabla \mu|_{\overline{S}} - \nabla \mu|_{\Omega \setminus S})\cdot n_{\Sigma}.$
\end{subequations}
\end{definition}

As singularities may develop in finite time, existence of strong solutions for general initial values is too optimistic, and we will focus our attention on a weak formulation of the MSR model. Within the rest of the paper, we will  \textbf{denote a weak solution of the CHR$_\epsilon$ model (\ref{pde:CHRnoElastic}) by $c_\epsilon$ with an associated chemical potential given by $\mu_\epsilon.$}
As is typically the program, we will prove compactness of $c_\epsilon$ in the appropriate function spaces, and then we will identify the limit as a weak solution of the MSR model, defined as follows.
\begin{definition}\label{def:BVsolnMSR}
Let $\Omega \subset \R^N$ be a bounded, open set with $C^2$ boundary. We say that $(c,\mu)$ is a \emph{BV(Bounded Variation)-solution of the MSR model} on $\Omega \times (0,T)$ if
\begin{align*}
& c\in   L^{\infty}(0,T;BV(\Omega;\{0,1\}))\cap C([0,T);H^1(\Omega)^*),\\
& \partial_t c\in L^{(2^\# - \delta)'}(0,T;H^{1}(\Omega)^*), \\
& c(0) = c_0 \in BV(\Omega;\{0,1\})\subset H^1(\Omega)^*, \\
& \mu \in L^2(0,T;H^1(\Omega)),
 \end{align*} and the optimal dissipation inequality
 \begin{equation}\label{def:MSR-dissipation}
\begin{aligned} 
I_0[c(T_*)] +\int_0^{T_*} \mathcal{A}_c(\mu) + \mathcal{A}^*_c(-\partial_t c) \, dt \leq I_0[c_0]
\end{aligned}
\end{equation}
holds for almost every $T_*$ in $(0,T)$,
where the chemical potential $\mu = \mu(x,t)$ is characterized by the Gibbs--Thomson relation (with $t$ suppressed)
\begin{equation}\label{def:GTrelationBV}
\int_\Omega {\rm div}(\mu \Psi) c \, dx  =  \sigma \int_\Omega \left( \mathbb{I} - \frac{Dc }{|Dc|} \otimes \frac{Dc }{|Dc|} \right):\nabla \Psi \, d|D c| ,
\end{equation}
which holds for almost every $t$ in $(0,T)$ and all $\Psi \in C^1(\Omega;\R^N)$ with $\Psi \cdot n_{\partial \Omega} = 0$ on $\partial \Omega$.
\end{definition} 

We note that the right hand side of (\ref{def:GTrelationBV}) is the first variation of the surface energy (\ref{def:energyPer}), which, when a surface is sufficiently regular, can be expressed in terms of the curvature \cite{MaggiBook}.
Furthermore, in parallel to our weak solution concept for the CHR$_\epsilon$ model: (\ref{def:GTrelationBV}) encapsulates the lower order relations (\ref{eqn:MSR3}) and (\ref{eqn:MSR6}), and the dissipation (\ref{def:MSR-dissipation}) covers the rest. To make these relations precise and to strengthen the weak solution concept's connection to the strong formulation, our first result shows that if $(c,\mu)$ is a sufficiently regular BV-solution, then it is a strong solution of the the MSR model as in Definition \ref{def:strongSolnMSR}.

\begin{thm} (Consistency)\label{thm:consistency}
Let $\Omega \subset \R^N$ be a bounded, open set with $C^2$ boundary. Suppose $(c,\mu)$ is a BV-solution of the MSR model, as in Definition \ref{def:BVsolnMSR}, such that $\mathcal{S} (t) := {\rm int}\{c = 1\}$ has boundary in $\Omega$ given by $\Sigma(t) := \partial \mathcal{S}(t) \cap \Omega$, which is a smoothly evolving surface, and $\mu \in C^\infty (\overline{\Omega \cap \{c=1\}}) \cap C^\infty (\overline{\Omega \cap \{c=0\}})$, then $\mathcal{S}$ is a strong solution of the MSR model as in Definition \ref{def:strongSolnMSR}.
\end{thm}

\begin{proof}[Proof of Theorem \ref{thm:consistency}]
We recover the lower order conditions within (\ref{eqn:strongMSR}). First, (\ref{eqn:MSR4}) comes for free as $\mu \in H^1(\Omega)$, and the trace operator is well defined on $\Sigma$ with no jump. Let $\Psi \in C^1(\Omega;\R^N)$ such that $\Psi \cdot n_{\p \Omega} = 0$ on $\partial \Omega$. Then we may decompose $\Psi$ into normal and tangential components along $\Sigma$ and use an integration by parts for the tangential vector field to find
\begin{equation}\label{eqn:innerVar}
\delta I_0[c](\Psi) = \sigma\int_{\Sigma} (\mathbb{I} - n_{\Sigma}\otimes n_{\Sigma}):\nabla \Psi \, d\mathcal{H}^{N-1}= -\sigma\left(\int_\Sigma  H_{\Sigma }\cdot \Psi\, d\mathcal{H}^{N-1} + \int_{\partial \Omega \cap \partial \Sigma} \Psi \cdot \tau \, d \mathcal{H}^{N-2} \right),
\end{equation} 
where $\delta I_0[c](\Psi)$ is a variation of the domain in the direction of $\Psi$ (see, e.g., \cite{MaggiBook}) and $\tau$ is the inward-pointing co-normal vector field on the boundary $\partial \Sigma$ within the tangent bundle of $\Sigma$ and we have suppressed dependence on $t$.
From the Gibbs--Thomson relation (\ref{def:GTrelationBV}) and using an integration by parts on $\mathcal{S}$ with $\Psi \cdot n_{\partial \Omega} = 0$, we may also compute that 
\begin{equation}\label{eqn:IBPGT}
\delta I_0[c](\Psi) = \int_{\Omega} {\rm div}(\mu \Psi)c  \, dx =  - \int_{\Sigma}\mu \Psi \cdot n_\Sigma \, d\mathcal{H}^{N-1}.
\end{equation}
Using compactly supported $\Psi$ we find that $\mu n_\Sigma= \sigma H_{\Sigma}$ on $\Sigma,$ concluding (\ref{eqn:MSR3}).

Combining (\ref{eqn:innerVar}) and (\ref{eqn:IBPGT}), we have $$ \int_{\partial \Omega \cap \partial \Sigma} \Psi \cdot \tau \, d \mathcal{H}^{N-2} = 0.$$
Choosing $\Psi$ as a local extension of the tangent vector $\tau_{\partial \Omega}(x_0) = \tau_{\partial \Omega}$ of $\partial \Omega$ near $x_0 \in \partial \Omega \cap \partial \Sigma,$ we see that $\tau \cdot \tau_{\partial \Omega} = 0$ on $\partial \Omega \cap \partial \Sigma$ for all $\tau_{\partial \Omega}$ tangent to $\partial \Omega.$ Consequently, $\tau  = n_{\partial \Omega},$ and we have (\ref{eqn:MSR6}).

We now turn to the higher order terms in the MSR model. As (\ref{eqn:MSR6}) is satisfied, the velocity $V_{\Sigma}$ can be extended to a function $\tilde V \in C^2(\Omega\times (0,T);\R^N)$ with $\tilde V \cdot n_{\partial \Omega} = 0.$ Consequently, we can compute the time derivative of the energy using the inner variation with respect to the diffeomorphism generated by $- \tilde V$ as
\begin{equation}\label{eqn:timeDerivativeSmooth}
\frac{d}{dt}I_0[c] = \sigma\int_{\Sigma} (\mathbb{I} - n_{\Sigma}\otimes n_\Sigma ):\nabla \tilde V \, d \mathcal{H}^{N-1} = \int_{\Omega}{\rm div}(\mu \tilde V) c \, dx= \la \partial_t c, \mu \ra_{H^{1}(\Omega)^*,H^1(\Omega)},
\end{equation}
where the second equality uses the Gibbs--Thomson relation (\ref{def:GTrelationBV}) and the last inequality follows from an integration by parts and (recall $c$ is a characteristic function)
\begin{equation}\label{eqn:surfaceFlow}
\partial_t c  =-\tilde V \cdot D c =  -\tilde V\cdot n_{\Sigma} \mathcal{H}^{N-1}|_{\Sigma}.
\end{equation}
We may localize the dissipation relation (\ref{def:MSR-dissipation}) for every $t$ in $(0,T)$ to find
\begin{equation}\label{eqn:localDissipation}
\frac{d}{dt}I_0[c] + \mathcal{A}_c(\mu) + \mathcal{A}^*_c(-\partial_t c)\leq 0.
\end{equation}
Using (\ref{eqn:timeDerivativeSmooth}) and (\ref{eqn:localDissipation}), we have equality in the Legendre--Fenchel inequality (\ref{eqn:legendreFenchelIneq}), and we conclude that
$$-\partial_t c \in \partial \mathcal{A}_c(\mu).$$ By (\ref{eqn:subdiffToBrelation}), this can be rewritten as
\begin{equation}\label{eqn:timeDerivStrongForm}
\la -\partial_t c, \xi \ra_{H^1(\Omega)^*,H^1(\Omega)} = \int_\Omega \nabla \mu \cdot \nabla \xi \, dx - \int_{\partial \Omega} R(c,\mu) \xi \, d\mathcal{H}^{N-1}
\end{equation}
for all $\xi \in H^1(\Omega).$
Testing (\ref{eqn:timeDerivStrongForm}) with $\xi \in C^1(\Omega)$ compactly supported away from $\Sigma,$ we recover (\ref{eqn:MSR1}) and (\ref{eqn:MSR5}).

For $\xi\in C^1(\Omega)$ compactly supported away from the boundary $\partial \Omega$, we use (\ref{eqn:surfaceFlow}) and an integration by parts to find
$$\int_\Sigma V_\Sigma \cdot n_\Sigma \xi \, d\mathcal{H}^{N-1} = - \int_\Sigma \jump{\partial_{n_{\Sigma}}\mu}\xi \, d \mathcal{H}^{N-1},$$
and localizing, we recover (\ref{eqn:MSR2}).
\end{proof}

We will prove existence of weak solutions to the MSR model (\ref{eqn:strongMSR}) by considering limit points of the solutions $c_\epsilon$ as $\epsilon \to 0$. 
Necessarily, our next result shows that solutions of the CHR$_\epsilon$ model are pre-compact in an appropriate topology. To show that the limiting function is a solution of the sharp interface model, we follow the general approach of Luckhaus and Sturzenhecker \cite{LucStu}, which requires the assumption of energy convergence. 
Developing this hypothesis, we recall the result of Modica \cite{Modica87} (see also \cite{ModicaMortola}), which shows that for a Lipschitz domain $\Omega\subset \R^N,$ the energy $I_\epsilon$ in (\ref{def:energy})
$\Gamma$-converges (see \cite{BraidesLocalMinNotes,DalMasoBook}) to the weighted perimeter functional 
\begin{equation}\label{def:energyPer}
I_0[c] : = \int_\Omega \sigma d |Dc| = \sigma{\rm Per} (\{c=1\}; \Omega), \quad \text{where }c \in BV( \Omega ; \{0,1\}) \text{ and }\sigma := \int_{0}^1 \sqrt{2f(s)}\, ds.
\end{equation}
The energy convergence assumption requires that
\begin{equation}\label{eqn:energyConvergence}
\int_0^T I_\epsilon[c_\epsilon(t)]\, dt \to \int_0^T I_0[c(t)]\, dt,
\end{equation}
where $c_\epsilon \to c$ as in Theorem \ref{thm:compactnessCHReps} below. This assumption implies that the sequence of solutions has an equipartition of energy in the limit, or more concretely, the $\liminf$ inequality following from $\Gamma$-convergence of $I_\epsilon \wkto_{\Gamma} I_0$ is in fact a limit at each point in time. Recently, such energy convergence hypotheses have been used to resolve questions in the context of (multiphase) mean curvature flow \cite{lauxSimon2018}, the Cahn--Hilliard equation with degenerate mobility~\cite{kroemerLaux2021}, and  density-constrained chemotaxis \cite{kimMelletWu2022}.

\begin{thm}\label{thm:compactnessCHReps}
Let $\Omega \subset \R^N$ be a bounded, open set with Lipschitz boundary. Assume hypotheses (\ref{ass:f}) to (\ref{ass:R3}) hold. For $T>0$, let $c_\epsilon$ be solutions of the CHR$_\epsilon$ model (\ref{pde:CHRnoElastic}) in the sense of Definition \ref{def:CHRweakSoln} such that $\limsup_{\epsilon\to 0}I_\epsilon[c_\epsilon(0)]<\infty$. Then there exists $$c \in L^\infty(0,T;BV(\Omega; \{0,1\}))\cap W^{1,(2^\# -\delta)'}(0,T;H^{1}(\Omega)^*),$$ such that up to a subsequence, the functions $c_\epsilon$ converge to $c$ in $L^2(0,T;L^{(2^\# - \delta)'}(\Omega))$. Furthermore, if $c_\epsilon$ satisfy the energy convergence relation (\ref{eqn:energyConvergence}), the traces $c_\epsilon \in L^\infty(0,T;L^{(2^\# - \delta)'}(\partial \Omega))$ converge in $L^2(0,T;L^{(2^\# - \delta)'}(\partial\Omega))$ to the trace of $c.$ Furthermore, the chemical potentials $\mu_\epsilon$ converge to $\mu$ weakly in $L^2(0,T;H^1(\Omega)).$
\end{thm}

\begin{thm}\label{thm:BVsolnExist}
Let $\Omega \subset \R^N$ be a bounded, open set with $C^2$ boundary, $T>0$, and $c_0 \in BV(\Omega ; \{0,1\})$. Assume hypotheses (\ref{ass:f}) to (\ref{ass:R3}) hold. Let $c_\epsilon$ converge to $c$ as in Theorem \ref{thm:compactnessCHReps} and satisfy the energy convergence relation (\ref{eqn:energyConvergence}) with well prepared initial data, in the sense that 
\begin{equation}\label{eqn:initWellPrep}
c_\epsilon (0) \to c_0\text{ in }L^1(\Omega)\quad \text{ and } \quad \lim_{\epsilon\to 0}I_\epsilon[c_\epsilon(0)] = I_0[c_0].
\end{equation} Then $c$ is a BV-solution of the MSR equation in the sense of Definition \ref{def:BVsolnMSR}.
\end{thm}

\begin{proof}[Proof of Theorem \ref{thm:compactnessCHReps}]

To deduce compactness in the bulk, we adapt an approach used in the context of Wasserstein gradient flows by Kroemer and the first author in \cite{kroemerLaux2021}. As in $\Gamma$-convergence for the Modica--Mortola energy, we will introduce the function $\phi$ associated with the geodesic distance for the two-well function. Then we will take advantage of the dissipative nature of the evolutionary process to gain compactness in time and space.

Recalling the growth in (\ref{ass:f}) and choosing $K > \max\{f(s):s\in (0,1)\}$, we introduce the function $\phi:\R\to \R$ as
\begin{equation}\label{def:phiInterfaceEnergy}
\phi(s) : = \int_{0}^{s} \sqrt{2\min\left\{f(t),\frac{1}{1+C}|s|^{\max\{1,2[(2^\# - \delta)'-1]\}}  + K\right\}}\, dt.
\end{equation}
Note $K$ is chosen to ensure that $\phi(1) = \sigma$ as in (\ref{def:energyPer}).

For $0<B<\infty$, define the energy $\mathcal{F}_{B}:L^{(2^\# -\delta)'}(\Omega)\to \R\cup \{\infty\}$ by
\begin{equation}\nonumber
\mathcal{F}_{B}[u]:=
\begin{cases}
\int_{\Omega} d|D (\phi \circ u)| & \text{ if }\left|\intbar_\Omega \phi \circ u \, dx \right|\leq B, \\
+\infty & \text{ otherwise.}
\end{cases} 
\end{equation}
Momentarily, we will apply a compactness theorem from \cite{rossiSavare2003} to the functional $\mathcal{F}_B$ acting on $c_\epsilon$. To apply the theorem, we must first show that $\mathcal{F}_B$ is 
\begin{enumerate}[label=\roman*)]
\item\label{item:1} measurable, 

\item lower semi-continuous,  

\item\label{item:3}  and has compact sublevel sets on $L^{(2^\# - \delta)'}(\Omega)$.
\end{enumerate}

As the constraints introduced in $\mathcal{F}_B$ are measurable, the functional is also measurable. For lower semi-continuity, if $u_i \to u$ in $L^{(2^\# - \delta)'}(\Omega)$ and $ \liminf_{i\to \infty} \mathcal{F}_B[u_i]<\infty$, then up to a subsequence, we have $ \lim_{i\to \infty} \mathcal{F}_B[u_i] = \liminf_{i\to \infty} \mathcal{F}_B[u_i]$. By the Poincar\'e inequality, compactness for $BV(\Omega),$ and pointwise convergence of $u_i \to u$, we have $\phi \circ u_i \to \phi \circ u$ in $L^1(\Omega)$. Lower semi-continuity of the total variation and continuity of the mass average under aforementioned convergence show that $\mathcal{F}_B$ is lower semi-continuous.  Similarly, the sublevel sets of $\mathcal{F}_B$ given by
$$S_{B,C} : = \{u\in L^{(2^\# - \delta)'}(\Omega):\mathcal{F}_B \leq C\} $$
are compact in $L^{(2^\# -\delta)'}(\Omega)$. Precisely, if $\mathcal{F}_B[u_i]\leq C$ for some collection of functions $\{u_i\}_{i\in \N} \subset L^{(2^\# - \delta)'}(\Omega),$ then as before, $\phi\circ u_i$ are uniformly bounded in $BV(\Omega)$ and have a subsequence converging in $L^1(\Omega)$ to $w \in BV(\Omega).$ As $\phi$ has $p$-growth with $p = (2^\# -\delta)'$ and continuous inverse, the Lebesgue dominated convergence theorem implies that $u_i \to u: = \phi^{-1}\circ w$ in $L^{(2^\# - \delta)'}(\Omega),$ as desired. This concludes the proof of properties \ref{item:1}-\ref{item:3}.

Turning to the collection of functions $c_\epsilon,$ we must show a tightness condition and uniform regularity in time. Note that from the growth condition (\ref{ass:f}), $$\int_\Omega |\phi\circ c_\epsilon| \, dx \leq \int_\Omega f(c_\epsilon)\, dx + C$$ for some fixed constant.
Consequently, given the dissipation relation (\ref{def:CHR-dissipation}), Young's inequality, and choosing $B\gg \sup_{\epsilon > 0} I_\epsilon [c_\epsilon(0)]$, we have that
\begin{equation}\label{eqn:RS03hyp1}
\int_0^T \mathcal{F}_{B}[c_\epsilon(t)]\, dt\leq T \sup_{t\in (0,T)}\int_\Omega \sqrt{2f(c_\epsilon)}\|\nabla c_\epsilon\|\, dx \leq T \sup_{t\in (0,T)}\int_\Omega \frac{1}{\epsilon}f(c_\epsilon) +  \frac{\epsilon}{2} \|\nabla c_\epsilon\|^2\, dx\leq TB.
\end{equation}

To gain uniform regularity in time, first note that by the coercivity of $G$ (\ref{ass:G2}) and a Poincar\'e inequality with trace, we have
\begin{equation}\label{eqn:muepsBdd}
\begin{aligned}
\mathcal{A}_c(\mu) \geq &\frac{1}{2}\int_\Omega \|\nabla \mu\|^2 \, dx+\frac{1}{C}\int_{\partial \Omega} |\mu|^2\, \mathcal{H}^{N-1} -C \geq  \frac{1}{C}\|\mu\|^2_{H^1(\Omega)}-C,
\end{aligned}
\end{equation}
from which we find $\int_0^T \mathcal{A}_c(\mu) \, dt >-\infty$.
Further, from the dissipation relation (\ref{def:CHR-dissipation}) and Young's inequality, we have $\|\phi \circ c_\epsilon\|_{L^\infty(0,T;BV(\Omega))} \leq C<\infty $ for some $C>0$ and all $\epsilon>0.$ Consequently, using continuity of the trace into $L^1(\partial \Omega)$ for functions in $BV(\Omega),$ we have $\|\phi \circ c_\epsilon\|_{L^\infty(0,T;L^1(\partial \Omega))} \leq C$, which by the growth condition (\ref{ass:f}) and definition (\ref{def:phiInterfaceEnergy}) further implies $\|c_\epsilon\|_{L^\infty(0,T;L^{(2^\# - \delta)'}(\partial \Omega))} \leq C =:\alpha$.
Applying Lemma \ref{lem:H1dualbdd} to (\ref{def:CHR-dissipation}), we recover a uniform (for $\epsilon>0$) bound on the time-derivative of $c_\epsilon$ given by 
\begin{equation}\label{eqn:unifTimeDerivBdd}
\|\partial_t c_\epsilon\|_{L^{(2^\#-\delta)'}(0,T;H^1(\Omega)^*)}\leq C<\infty.
\end{equation}

Finally, with (\ref{eqn:RS03hyp1}), (\ref{eqn:unifTimeDerivBdd}), and properties \ref{item:1}-\ref{item:3} of the functional $\mathcal{F}_B$, the hypotheses of \cite[Theorem 2]{rossiSavare2003}, a variant of the Aubin--Lions--Simon compactness theorem, are satisfied, and up to a subsequence, $c_\epsilon$ converge in measure to a function $c$, meaning,
\begin{equation}\nonumber
\limsup_{\epsilon\to 0} \mathcal{L}^1(\{t: \|c_\epsilon(t)-c(t)\|_{L^{(2^\# - \delta)'}(\Omega)}>\eta\}) = 0\quad \text{ for all }\eta>0.
\end{equation}
Up to a further subsequence, we may assume that $c_\epsilon(t) \to c(t)$ in $L^{(2^\# - \delta)'}(\Omega)$ for almost every $t\in (0,T),$ which by Lebesgue's dominated convergence theorem, (\ref{def:CHR-dissipation}), and (\ref{ass:f}) gives $c_\epsilon\to c$ in $L^2(0,T;L^{(2^\# - \delta)'}(\Omega))$. From the dissipation relation (\ref{def:CHR-dissipation}), for almost every $t$ in $(0,T),$ $c$ must belong to $L^1(\Omega ; \{0,1\})$. From the lower semi-continuity of $\mathcal{F}_B$ and the lower semi-continuity of the norm for the time derivative, we conclude that $c\in L^\infty(0,T;BV(\Omega;\{0,1\}))\cap W^{1,(2^\#-\delta)'}(0,T;H^1(\Omega)^*).$

To obtain convergence of the traces, define the functions $w_\epsilon := \phi \circ c_\epsilon$ which, by compactness of $c_\epsilon$, converge to $w:=\phi\circ c = \sigma c$ in $L^2(0,T;L^1(\Omega))$. By lower semi-continuity of the total variation and Fatou's lemma, we have
\begin{equation}\nonumber
 \int_0^T I_0(c)\, dt  = \int_0^T \int_\Omega d|D w|\, dt \leq \liminf_{\epsilon\to 0}\int_0^T \int_\Omega d|D w_\epsilon|\, dt.
\end{equation} 
By the energy convergence (\ref{eqn:energyConvergence}) and Young's inequality, we obtain the matching upper bound
\begin{equation}\nonumber
\limsup_{\epsilon\to 0}\int_0^T \int_\Omega d|D w_\epsilon|\, dt \leq \limsup_{\epsilon\to 0} \int_0^T I_\epsilon[c_\epsilon]\, dt = \int_0^T I_0(c)\, dt.
\end{equation}
The above equations show that $w_\epsilon$ converges to $w$ \textit{strictly} in $BV(\Omega)$, where here by strict convergence it is meant
\begin{equation}\label{eqn:strictConv}
w_\epsilon \to w \text{ in }L^1(0,T;L^1(\Omega)) \text{ and }\int_0^T \int_\Omega d|D w_\epsilon|\, dt \to \int_0^T \int_\Omega d|D w|\, dt.
\end{equation} 
By \cite[Theorem 3.88]{AmbrosioFuscoPallara}, the trace map from $BV(\Omega)\to L^1(\partial\Omega)$ is continuous if $BV(\Omega)$ is equipped with the topology from strict convergence, and up to minor modification of the proof, the result still holds under strict convergence in the sense of (\ref{eqn:strictConv}). Consequently, $w_\epsilon \to w$ in $L^1(0,T;L^1(\partial \Omega))$, but as (\ref{def:CHR-dissipation}) and the trace estimate $\|u\|_{L^1(\partial \Omega)}\leq \|u\|_{BV(\Omega)}$ imply $w_\epsilon$ are uniformly bounded in $L^\infty(0,T;L^1(\partial\Omega))$, convergence may be improved to $L^2(0,T;L^1(\partial\Omega)).$ As $\phi$ has continuous inverse and $p$-growth with $p\geq (2^\# - \delta)'$, we conclude the convergence of traces of $c_\epsilon$ as in the theorem statement.

Lastly, passing to a subsequence if necessary, $\mu_\epsilon$ weakly converges to $\mu$ in $L^2(0,T;H^1(\Omega))$ by (\ref{eqn:muepsBdd}).
\end{proof}

\begin{proof}[Proof of Theorem \ref{thm:BVsolnExist}]
\textbf{Step 1: Dissipation.}
We begin by obtaining the dissipation in the limit. First note that $\mu_\epsilon $ in $L^2(0,T;L^2(\partial \Omega))$ generates a Young measure $\alpha_{(x,t)}$ with $\int_\R y \, d\alpha_{(x,t)} = \mu(x,t)$ for $\mathcal{H}^N$-a.e. $(x,t)\in \partial\Omega \times (0,T)$. By Theorem \ref{thm:compactnessCHReps}, $c_\epsilon$ in $L^2(0,T;L^1(\partial\Omega))$ generates a Young measure $\gamma_{(x,t)} = \delta_{c(x,t)}.$ By the fundamental theorem for Young measures applied to $\gamma_{(x,t)} \otimes \alpha_{(x,t)}$ and convexity of $-G(c,\cdot)$ coupled with Jensen's inequality, we find 
\begin{equation}\nonumber
\begin{aligned}
\liminf_{\epsilon \to 0} \int_0^T  \int_{\partial \Omega} -G(c_\epsilon,\mu_\epsilon) \ d\mathcal{H}^{N-1}  dt & = \int_0^T  \int_{\partial \Omega} \int_{\R} -G(c,y) \, d\alpha_{(x,t)}(y)\, d\mathcal{H}^{N-1} dt \\
& \geq \int_0^T  \int_{\partial \Omega}  -G(c,\mu)  d\mathcal{H}^{N-1}  dt.
\end{aligned}
\end{equation}
The above relation, lower semi-continuity of the $L^2$ norm, and the definition (\ref{def:functA}) show that
\begin{equation}\label{eqn:ALSC}
\begin{aligned}
\liminf_{\epsilon \to 0} \int_0^T \mathcal{A}_{c_\epsilon}(\mu_{\epsilon})\, dt \geq \int_0^T \mathcal{A}_{c}(\mu)\, dt.
\end{aligned}
\end{equation}

To prove lower semi-continuity of the conjugate $\mathcal{A}^*$, we consider a dense subset $\{v_i\}_{i\in \N} \subset H^1(\Omega)\cap C(\bar \Omega).$ Then by (\ref{ass:G1}) and Theorem \ref{thm:compactnessCHReps}, $$\int_{0}^T \left| \mathcal{A}_{c_\epsilon}(v_i) -\mathcal{A}_{c}(v_i) \right| dt \to 0.$$ We also remark that $$v^* \mapsto \int_0^T \max_{i\leq k}\left\{ \la v^*, v_i \ra_{H^1(\Omega)^*,H^1(\Omega)}  - \mathcal{A}_{c}(v_i)\right\} dt$$ is a convex functional over $L^{(2^\#-\delta)'}(0,T;H^1(\Omega)^*)$. Putting these comments together and using lower semi-continuity of convex functionals under weak convergence, we have 
\begin{equation}\nonumber
\begin{aligned}
\liminf_{\epsilon \to 0} \int_0^T \mathcal{A}_{c_\epsilon}^*(-\partial_t c_{\epsilon})\, dt  \geq & \liminf_{\epsilon \to 0} \int_0^T \max_{i\leq k}\left\{ \la - \partial_t c_\epsilon, v_i \ra_{H^1(\Omega)^*,H^1(\Omega)}  - \mathcal{A}_{c_\epsilon}(v_i)\right\} dt \\
 \geq & \liminf_{\epsilon\to 0}\int_0^T \max_{i\leq k}\left\{ \la - \partial_t c_\epsilon, v_i \ra_{H^1(\Omega)^*,H^1(\Omega)}  - \mathcal{A}_{c}(v_i) \right\} dt \\
& - \limsup_{\epsilon\to 0} \int_0^T \max_{i\leq k}\{| \mathcal{A}_{c}(v_i) -  \mathcal{A}_{c_\epsilon}(v_i)|\} \, dt \\
\geq & \int_0^T \max_{i\leq k}\left\{ \la - \partial_t c, v_i \ra_{H^1(\Omega)^*,H^1(\Omega)}  - \mathcal{A}_{c}(v_i) \right\} dt.
\end{aligned}
\end{equation} 
Then we may apply the monotone convergence theorem as $k\to \infty$ to find
\begin{equation}\label{eqn:AstarLSC}
\begin{aligned}
\liminf_{\epsilon \to 0} \int_0^T \mathcal{A}_{c_\epsilon}^*(-\partial_t c_{\epsilon})\, dt & \geq  \int_0^T \mathcal{A}_{c}^*(-\partial_t c)\, dt.
\end{aligned}
\end{equation}

By convergence of $c_\epsilon(t)$ to $c(t)$ in $L^1(\Omega)$ for almost every $t$ in $(0,T)$ and the $\Gamma-\liminf$ inequality for the energy $I_\epsilon$, we have 
\begin{equation}\label{eqn:EnergyLSC}
I_0[c(t)] \leq \liminf_{\epsilon\to 0} I_\epsilon [c_\epsilon(t)]
\end{equation} 
for almost every $t$ in $(0,T)$. With the well-preparedness of the initial condition from (\ref{eqn:initWellPrep}), by (\ref{eqn:ALSC}), (\ref{eqn:AstarLSC}), and (\ref{eqn:EnergyLSC}), we conclude the dissipation relation (\ref{def:MSR-dissipation}).

\textbf{Step 2: Gibbs--Thomson relation.}
 To recover the Gibbs--Thomson relation (\ref{def:GTrelationBV}), we show that the higher-order nonlinearity arising from the metric interacts well with the convergence arising from compactness.
We begin by showing that an approximate version of the Gibbs--Thomson relation, involving the stress tensor, still holds in our setting; this follows a classic computation of Luckhaus and Modica \cite{LuckhausModica}. Precisely, we show
\begin{equation}\label{eqn:approxGT}
 \int_\Omega \left[\Big(\frac{\epsilon}{2} \|\nabla c_\epsilon\|^2 +\frac{1}{\epsilon}f(c_\epsilon)\Big) \mathbb{I} - \epsilon \nabla c_\epsilon \otimes \nabla c_\epsilon \right]:\nabla \Psi \, dx = \int_\Omega {\rm div}(\mu_\epsilon \Psi) c_\epsilon \, dx 
\end{equation}
for all $\Psi \in C^1(\overline{\Omega};\R^N)$ with $\Psi \cdot n_{\partial \Omega}  = 0$ on $\partial \Omega,$ where $n_{\partial \Omega}$ is the inward normal.

Fix $\Psi \in C^1(\overline{\Omega};\R^N)$ with $\Psi \cdot n_{\partial \Omega} = 0$ on $\partial \Omega$. Note that by elliptic regularity, $c_\epsilon$ belongs to $H^2(\Omega)$ for almost every $t$ in $(0,T)$, and so testing (\ref{def:CHR-ELeqn}) with $\Psi \cdot  \nabla c_\epsilon$, we compute
\begin{align*}
&-\int_\Omega {\rm div}(\mu_\epsilon \Psi) c_\epsilon \, dx -\int_{\partial \Omega} \mu_\epsilon c_\epsilon \Psi \cdot n_{\partial \Omega} \, d\mathcal{H}^{N-1} \\
& = \int_\Omega \mu_\epsilon (\Psi \cdot \nabla c_\epsilon)\, dx \\
&=\int_\Omega \epsilon \nabla\left(\Psi \cdot \nabla c_\epsilon\right)\cdot  \nabla c_\epsilon +\frac{1}{\epsilon}\Psi \cdot  \nabla c_\epsilon f'(c_\epsilon) \, dx \\
& = \int_\Omega \epsilon \nabla\left(\Psi \cdot \nabla c_\epsilon\right)\cdot  \nabla c_\epsilon +\frac{1}{\epsilon}\Psi \cdot  \nabla f(c_\epsilon) \, dx \\
& = \int_\Omega \epsilon \nabla \Psi : \nabla c_\epsilon \otimes \nabla c_\epsilon  + \epsilon\sum_{i,j}(\Psi_j \partial_i\partial_j c_\epsilon)\partial_i c_\epsilon -\frac{1}{\epsilon}\nabla \Psi :   f(c_\epsilon)\mathbb{I} \, dx - \int_{\partial \Omega} \Psi \cdot n_{\partial \Omega} \frac{1}{\epsilon} f(c_\epsilon) \, d \mathcal{H}^{N-1} \\
& = \int_\Omega \nabla \Psi :\left[\epsilon \nabla c_\epsilon \otimes \nabla c_\epsilon   -   \left(\frac{1}{\epsilon}f(c_\epsilon) + \frac{\epsilon}{2} \|\nabla c_\epsilon\|^2\right)\mathbb{I}\right] \, dx  \\
& \quad \quad - \int_{\partial \Omega} \Psi \cdot n_{\partial \Omega} \left(\frac{1}{\epsilon}f(c_\epsilon)+\frac{\epsilon}{2}\|\nabla c_\epsilon\|^2\right) \, d \mathcal{H}^{N-1}.
\end{align*}
As the boundary terms vanish, we have (\ref{eqn:approxGT}). 

We will integrate (\ref{eqn:approxGT}) in time and pass to the limit. First note that the measures generated by the energy also converge locally, that is, for $\phi$ as in (\ref{def:phiInterfaceEnergy}), we have
\begin{equation}\label{eqn:approxSurfaceWeakConv}
E_\epsilon : = \left(\frac{1}{\epsilon} f(c) + \frac{\epsilon}{2}\|\nabla c\|^2\right)\mathcal{L}^N\llcorner \Omega \wksto \sigma |D c|\llcorner \Omega\quad \text{and} \quad |D [ \phi \circ c_\epsilon]| \wksto \sigma |D c|\llcorner \Omega \ \ \text{ in }C_c(\overline{\Omega} \times [0,T])^*.
\end{equation}
To see this, letting $\xi \in C(\overline{\Omega}\times [0,T])$ be such that $0<\xi<1$, we use the $\Gamma-\liminf$ inequality and Fatou's lemma to compute
\begin{equation}\nonumber
\begin{aligned}
\sigma \int_0^T\int_\Omega \xi d|Dc|\, dt &= \sigma\int_0^T\int_{\R^+}\int_\Omega \chi_{\{\xi>s\}} \, d|Dc| \, ds \, dt  \\
&\leq \liminf_{\epsilon\to 0} \int_0^T\int_{\R^+}\int_\Omega \chi_{\{\xi>s\}} \, dE_\epsilon \, ds \, dt = \liminf_{\epsilon\to 0}\int_0^T\int_\Omega \xi \, dE_\epsilon \, dt .
\end{aligned}
\end{equation}
As the same inequality holds for $(1-\xi)$, we can use the energy convergence hypothesis (\ref{eqn:energyConvergence}) to conclude $\limsup_{\epsilon \to 0} \int_\Omega \xi \, dE_\epsilon \leq \sigma \int_\Omega \xi d|Dc|,$ which up to rescaling and translation of $\xi$ concludes the first relation in (\ref{eqn:approxSurfaceWeakConv}); and the second follows similarly.

The convergence above will be sufficient to pass to the limit for the left-hand side of (\ref{eqn:approxGT}), excluding the nonlinear tensor. For this remaining term, define $$n_\epsilon : = \frac{\nabla (\phi \circ c_\epsilon)}{\|\nabla (\phi \circ c_\epsilon)\|},$$ and rewrite the time-integrated tensor term in (\ref{eqn:approxGT}) as  
\begin{equation}\label{eqn:tensoredTerm}
 A_\epsilon:= \int_0^T\int_\Omega  n_\epsilon \otimes n_\epsilon:\nabla \Psi \, \epsilon  \|\nabla c_\epsilon\|^2 \, dx \, dt.
 \end{equation}
Given the equiparition following from the energy convergence assumption, $A_\epsilon$ satisfies the following equality up to a lower order term:
\begin{equation}\label{eqn:Aeps}
 A_\epsilon = \int_0^T\int_\Omega  n_\epsilon \otimes n_\epsilon:\nabla \Psi \,  dE_\epsilon \, dt + o_{\epsilon\to 0}(1) = \int_0^T\int_\Omega  n_\epsilon \otimes n_\epsilon:\nabla \Psi \,  d|D[\phi \circ c_\epsilon]| \, dt + o_{\epsilon\to 0}(1).
 \end{equation}
To pass to the limit in (\ref{eqn:tensoredTerm}), we pass $\epsilon \to 0$ on the right-hand side of the above equation.
As a technical aid, we fix one of the normal directions by introducing an auxiliary vector field $n^* \in C(\overline{\Omega}\times [0,T];\R^N)$ (to be chosen later) and, with $n:= Dc/|Dc|$, compute 
\begin{align}
\notag&\left| \int_0^T\int_\Omega  n_\epsilon \otimes n_\epsilon:\nabla \Psi \,  d|D[\phi \circ c_\epsilon]| dt - \sigma \int_0^T\int_\Omega  n \otimes n:\nabla \Psi \,  d|Dc|dt \right|  \\
\notag&\leq \left| \int_0^T\int_\Omega  n^* \otimes n_\epsilon:\nabla \Psi \,  d|D[\phi \circ c_\epsilon]| dt - \sigma\int_0^T\int_\Omega  n^* \otimes n:\nabla \Psi \,  d|Dc|dt \right| \\
\notag& \quad + \left| \int_0^T\int_\Omega (n_\epsilon - n^*) \otimes n_\epsilon:\nabla \Psi \,  d|D[\phi \circ c_\epsilon]| dt\right| 
+ \left| \sigma\int_0^T\int_\Omega (n - n^*) \otimes n :\nabla \Psi \,  d|Dc| dt\right| \\
\notag&\leq \left| \int_0^T\int_\Omega  \la n^*, \nabla \Psi \nabla(\phi \circ c_\epsilon)\ra\,  dx \,  dt - \sigma \int_0^T\int_\Omega \la n^* , \nabla \Psi  d[Dc]\ra \, dt \right| \\
& \quad + \left| \int_0^T\int_\Omega (n_\epsilon - n^*) \otimes n_\epsilon:\nabla \Psi \,  d|D[\phi \circ c_\epsilon]| dt\right| \label{eqn:normalCalc} \quad + \left| \sigma\int_0^T\int_\Omega (n - n^*) \otimes n :\nabla \Psi \,  d|Dc| dt\right|.
\end{align}
The first term on the right vanishes as $\epsilon \to 0$ by weak star convergence of $D[\phi \circ c_\epsilon]$ to $\sigma Dc$ on $\overline{\Omega}\times [0,T]$. The remaining two terms are controlled by
 $$C \left(\int_0^T\int_\Omega |n_\epsilon - n^*|^2 d|D[\phi \circ c_\epsilon]| \, dt \right)^{1/2} + C \left(\int_0^T\int_\Omega |n - n^*|^2 d|Dc| \, dt \right)^{1/2} ,$$
where $C > 0$ depends on $\Psi$ and the $L^\infty$ bounds on the energy following from (\ref{def:CHR-dissipation}).
By \cite[Lemma 4.8]{kroemerLaux2021}, the first of these terms converges to the second (times $\sigma$) as $\epsilon\to 0.$ Applying density, we can choose $n^* \in C(\overline{\Omega}\times [0,T];\R^N)$ such that
$$ \left(\int_0^T\int_\Omega |n - n^*|^2 d|Dc| \, dt \right)^{1/2} <\delta,$$ where $\delta>0$ is arbitrary. Consequently, the right-hand side of (\ref{eqn:normalCalc}) tends to zero as $\epsilon\to 0,$ and recalling (\ref{eqn:Aeps}), we find that
$$A_\epsilon \to  \sigma \int_0^T\int_\Omega  n \otimes n:\nabla \Psi \,  d|Dc|dt.$$
Recalling definition (\ref{eqn:tensoredTerm}), the convergence in (\ref{eqn:approxSurfaceWeakConv}), and the compactness from Theorem \ref{thm:compactnessCHReps}, we integrate (\ref{eqn:approxGT}) in time, pass to the limit, then localize in time to recover (\ref{def:GTrelationBV}), thereby concluding the theorem.
\end{proof}

\bibliographystyle{amsplain}
\bibliography{CHR_sharp_interface}

\providecommand{\bysame}{\leavevmode\hbox to3em{\hrulefill}\thinspace}
\providecommand{\MR}{\relax\ifhmode\unskip\space\fi MR }
\providecommand{\MRhref}[2]{%
  \href{http://www.ams.org/mathscinet-getitem?mr=#1}{#2}
}
\providecommand{\href}[2]{#2}
\begin{thebibliography}{10}

\bibitem{abelsLengeler2014}
H.~Abels and D.~Lengeler, \emph{On sharp interface limits for diffuse interface
  models for two-phase flows}, Interfaces Free Bound. \textbf{16} (2014),
  no.~3, 395--418.

\bibitem{AbelsRoeger}
H.~Abels and M.~R\"oger, \emph{Existence of weak solutions for a non-classical
  sharp interface model for a two-phase flow of viscous, incompressible
  fluids}, Ann. Inst. H. Poincar\'e Anal. Non Lin\'eaire \textbf{26} (2009),
  no.~6, 2403--2424.

\bibitem{Abels2015SharpIL}
H.~Abels and S.~Schaubeck, \emph{Sharp interface limit for the
  {C}ahn--{L}arch\'e system}, Asymptot. Anal. \textbf{91} (2015), 283--340.

\bibitem{AlikakosBatesChen}
N.~Alikakos, P.~Bates, and X.~Chen, \emph{Convergence of the {Cahn-Hilliard}
  equation to the {Hele-Shaw} model}, Arch. Rational Mech. Anal. \textbf{128}
  (1994), no.~2, 165--205.

\bibitem{AmbrosioFuscoPallara}
L.~Ambrosio, N.~Fusco, and D.~Pallara, \emph{Functions of bounded variation and
  free discontinuity problems}, Oxford Mathematical Monographs, The Clarendon
  Press, Oxford University Press, New York, 2000.

\bibitem{AGS-GradFlows}
L.~Ambrosio, N.~Gigli, and G.~Savar{\'e}, \emph{Gradient flows in metric spaces
  and in the space of probability measures}, Birkh{\"a}user Basel, 2008.

\bibitem{ambrosio-tortorelli-1990}
L.~Ambrosio and V.~M. Tortorelli, \emph{Approximation of functionals depending
  on jumps by elliptic functionals via {$\Gamma$}-convergence}, Comm. Pure
  Appl. Math. \textbf{43} (1990), no.~8, 999--1036.

\bibitem{Bazant-Theory2013}
M.~Z. Bazant, \emph{Theory of chemical kinetics and charge transfer based on
  nonequilibrium thermodynamics}, Accounts of Chemical Research \textbf{46}
  (2013).

\bibitem{BraidesLocalMinNotes}
A.~Braides, \emph{Local minimization, variational evolution and
  {$\Gamma$}-convergence}, Lecture Notes in Mathematics, vol. 2094, Springer,
  Cham, 2014.

\bibitem{burch2009size}
D.~Burch and M.~Z. Bazant, \emph{Size-dependent spinodal and miscibility gaps
  for intercalation in nanoparticles}, Nano Letters \textbf{9} (2009), no.~11,
  3795--3800.

\bibitem{chen1996}
X.~Chen, \emph{Global asymptotic limit of solutions of the {C}ahn--{H}illiard
  equation}, J. Differential Geom. \textbf{44} (1996), no.~2, 262--311.

\bibitem{cogswell2012coherency}
D.~A. Cogswell and M.~Z. Bazant, \emph{Coherency strain and the kinetics of
  phase separation in {LiFePO$_4$} nanoparticles}, ACS Nano \textbf{6} (2012),
  no.~3, 2215--2225.

\bibitem{ContiFonsecaLeoni-gammConv2grad}
S.~Conti, I.~Fonseca, and G.~Leoni, \emph{A {G}amma-convergence result for the
  two gradient theory of phase transitions}, Comm. Pure Appl. Math. \textbf{55}
  (2001).

\bibitem{Dal2015-Comp}
H.~Dal and C.~Miehe, \emph{Computational electro-chemo-mechanics of lithium-ion
  battery electrodes at finite strains}, Comput. Mech. \textbf{55} (2015),
  no.~2, 303--325.

\bibitem{DalMasoBook}
G.~Dal~Maso, \emph{An introduction to {$\Gamma$}-convergence}, Progress in
  Nonlinear Differential Equations and their Applications, 8, Birkh\"auser
  Boston, Inc., Boston, MA, 1993.

\bibitem{DeGiorgi1980}
E.~De~Giorgi, A.~Marino, and M.~Tosques, \emph{Problemi di evoluzione in spazi
  metrici e curve di massima pendenza}, Atti Accad. Naz. Lincei Cl. Sci. Fis.
  Mat. Natur. Rend. Lincei (9) Mat. Appl. \textbf{68} (1980), no.~3, 180--187.

\bibitem{elliott-CHdegenMobility}
C.~M. Elliott and H.~Garcke, \emph{On the {Cahn–Hilliard} equation with
  degenerate mobility}, SIAM J. Math. Anal. \textbf{27} (1996), no.~2,
  404--423.

\bibitem{FJMrigid}
G.~Friesecke, R.~D. James, and S~M{\"u}ller, \emph{A theorem on geometric
  rigidity and the derivation of nonlinear plate theory from three-dimensional
  elasticity}, Comm. Pure Appl. Math. \textbf{55} (2002), no.~11, 1461--1506.

\bibitem{garcke-CHdynamicBC}
H.~Garcke and P.~Knopf, \emph{Weak solutions of the {Cahn-Hilliard} system with
  dynamic boundary conditions: A gradient flow approach}, SIAM J. Math. Anal.
  \textbf{52} (2020), no.~1, 340--369.

\bibitem{Hensel2021l}
S.~Hensel and T.~Laux, \emph{A new varifold solution concept for mean curvature
  flow: {C}onvergence of the {A}llen--{C}ahn equation and weak-strong
  uniqueness}, arXiv preprint (2021),
  \href{https://arxiv.org/abs/2109.04233}{arXiv:2109.04233}.

\bibitem{henselStinson-weakSolnMS}
S.~Hensel and K.~Stinson, \emph{Weak solutions of {M}ullins-{S}ekerka flow as a
  {H}ilbert space gradient flow}, arXiv preprint (2022),
  \href{https://arxiv.org/abs/2206.08246}{arxiv:2206.08246}.

\bibitem{ilmanen}
T.~Ilmanen, \emph{Convergence of the {A}llen--{C}ahn equation to {B}rakke's
  motion by mean curvature}, J. Differential Geom. \textbf{38} (1993), no.~2,
  417--461.

\bibitem{kimMelletWu2022}
I.~Kim, A.~Mellet, and Y.~Wu, \emph{Density-constrained chemotaxis and
  {H}ele-{S}haw flow}, arXiv preprint (2022),
  \href{https://arxiv.org/abs/2204.11917}{arxiv:2204.11917}.

\bibitem{kraus2017}
C.~Kraus and A.~Roggensack, \emph{Existence of weak solutions for the
  {Cahn--Hilliard} reaction model including elastic effects and damage}, J.
  Partial Differential Equations \textbf{30} (2017), 111--145.

\bibitem{kroemerLaux2021}
M.~Kroemer and T.~Laux, \emph{The {H}ele-{S}haw flow as the sharp interface
  limit of the {C}ahn--{H}illiard equation with disparate mobilities}, arXiv
  preprint (2021), \href{https://arxiv.org/abs/2111.14505}{arXiv:2111.14505}.

\bibitem{lauxSimon2018}
T.~Laux and T.~M. Simon, \emph{Convergence of the {Allen--Cahn} equation to
  multiphase mean curvature flow}, Comm. Pure Appl. Math. \textbf{71} (2018),
  no.~8, 1597--1647.

\bibitem{Le2008}
N.~Q. Le, \emph{{A} {G}amma-convergence approach to the {C}ahn--{H}illiard
  equation}, Calc. Var. Partial Differential Equations \textbf{32} (2008),
  no.~4, 499--522.

\bibitem{LuckhausModica}
S.~Luckhaus and L.~Modica, \emph{The {G}ibbs-{T}hompson relation within the
  gradient theory of phase transitions}, Arch. Rational Mech. Anal.
  \textbf{107} (1989), no.~1, 71--83. \MR{1000224 (90k:49041)}

\bibitem{LucStu}
S.~Luckhaus and T.~Sturzenhecker, \emph{Implicit time discretization for the
  mean curvature flow equation}, Calc. Var. Partial Differential Equations
  \textbf{3} (1995), no.~2, 253--271.

\bibitem{MaggiBook}
F.~Maggi, \emph{Sets of finite perimeter and geometric variational problems: An
  introduction to geometric measure theory}, Cambridge Studies in Advanced
  Mathematics, vol. 135, Cambridge University Press, 2012.

\bibitem{meca_munch_wagner_2018}
E.~Meca, A.~M{\"u}nch, and B.~Wagner, \emph{Sharp-interface formation during
  lithium intercalation into silicon}, European J. Appl. Math. \textbf{29}
  (2018), no.~1, 118–145.

\bibitem{melchionnaRocca2017}
S.~Melchionna and E.~Rocca, \emph{Varifold solutions of a sharp interface limit
  of a diffuse interface model for tumor growth}, Interfaces Free Bound.
  \textbf{19} (2017), no.~4, 571--590.

\bibitem{mielke-doublyNonlinearEvo}
A.~Mielke, R.~Rossi, and G.~Savar{\'e}, \emph{Nonsmooth analysis of doubly
  nonlinear evolution equations}, Calc. Var. Partial Differential Equations
  \textbf{46} (2013), 253–310.

\bibitem{Modica87}
L.~Modica, \emph{The gradient theory of phase transitions and the minimal
  interface criterion}, Arch. Rational Mech. Anal. \textbf{98} (1987), no.~2,
  123--142.

\bibitem{ModicaMortola}
L.~Modica and S.~Mortola, \emph{Un esempio di {$\Gamma$}-convergenza}, Boll.
  Un. Mat. Ital. B (5) \textbf{14} (1977), no.~1, 285--299.

\bibitem{O_Connor_2016}
D.~T. O'Connor, M.~J. Welland, W.~K. Liu, and P.~W. Voorhees, \emph{Phase
  transformation and fracture in single {L}i$_x${FePO}$_4$ cathode particles: a
  phase-field approach to {L}i-ion intercalation and fracture}, Modelling and
  Simulation in Materials Science and Engineering \textbf{24} (2016), no.~3.

\bibitem{onsager31}
L.~Onsager, \emph{Reciprocal relations in irreversible processes. {II}.}, Phys.
  Rev. \textbf{38} (1931), 2265--2279.

\bibitem{rockafellar-convex}
R.~T. Rockafellar, \emph{Convex analysis}, Princeton University Press, 1970.

\bibitem{Roeger2005}
M.~R\"{o}ger, \emph{Existence of weak solutions for the {M}ullins--{S}ekerka
  flow}, {SIAM} J. Math. Anal. \textbf{37} (2005), no.~1, 291--301.

\bibitem{rossiSavare2003}
R.~Rossi and G.~Savar\'e, \emph{Tightness, integral equicontinuity and
  compactness for evolution problems in {Banach} spaces}, Ann. Sc. Norm. Super.
  Pisa Cl. Sci. \textbf{Ser. 5, 2} (2003), no.~2, 395--431.

\bibitem{sandierSerfaty-gammaGrad}
E.~Sandier and S.~Serfaty, \emph{Gamma-convergence of gradient flows with
  applications to {G}inzburg--{L}andau}, Comm. Pure Appl. Math. \textbf{57}
  (2004), no.~12, 1627--1672.

\bibitem{Serfaty2011}
S.~Serfaty, \emph{Gamma-convergence of gradient flows on {H}ilbert and metric
  spaces and applications}, Discrete Contin. Dyn. Syst. - A \textbf{31} (2011),
  no.~4, 1427--1451.

\bibitem{simonCompact}
J.~Simon, \emph{Compact sets in the space {$L^p(O,T; B)$}}, Ann. Mat. Pura
  Appl. \textbf{146} (1986), 65--96.

\bibitem{singh2008intercalation}
G.~K. Singh, G.~Ceder, and M.~Z. Bazant, \emph{Intercalation dynamics in
  rechargeable battery materials: general theory and phase-transformation waves
  in {LiFePO$_4$}}, Electrochimica Acta \textbf{53} (2008), no.~26, 7599--7613.

\bibitem{stinson-phd}
K.~Stinson, \emph{Analysis of a variational model for lithium-ion batteries},
  {PhD} dissertation, Carnegie Mellon University, 2021.

\bibitem{stinsonLiBatteryGamma}
\bysame, \emph{On {$\Gamma$-}convergence of a variational model for lithium-ion
  batteries}, Arch. Ration. Mech. Anal \textbf{240} (2021), 1--50.

\bibitem{stinsonLiBatteryExpBC}
\bysame, \emph{Existence for a {C}ahn--{H}illiard model for lithium-ion
  batteries with exponential-growth boundary conditions}, arXiv preprint
  (2022), \href{https://arxiv.org/abs/2010.08533}{arxiv:2010.08533}.

\bibitem{Bazant-PhaseSepDyn2014}
Y.~Zeng and M.~Z. Bazant, \emph{Phase separation dynamics in isotropic
  ion-intercalation particles}, SIAM J. Appl. Math. \textbf{74} (2013).

\end{thebibliography}

\end{document}